\documentclass[12pt]{amsart}
\usepackage{cancel}
\usepackage{amssymb}
\usepackage{hyperref}
\usepackage{enumitem}
\usepackage{amsmath}
\usepackage{latexsym}
\usepackage{extarrows}
\usepackage{enumitem}
\usepackage{txfonts}
\usepackage{mathtools}
\usepackage{bm}
\usepackage{tikz-cd}
\usepackage{xcolor}
\usepackage{quiver}
\usepackage{comment}
\usepackage{graphics}
\usepackage{float}
\usepackage{relsize}
\usepackage{bm}
\usepackage[normalem]{ulem}
\makeatletter
\@namedef{subjclassname@2020}{%
  \textup{2020} Mathematics Subject Classification}
\makeatother
\usepackage[T1]{fontenc}

\hypersetup{
    hidelinks,
    linktoc=page,
    }

\counterwithin{figure}{section}

\numberwithin{equation}{section}

\theoremstyle{plain}

\newtheorem{theorem}{Theorem}[section]
\newtheorem{proposition}[theorem]{Proposition}

\newtheorem{lemma}[theorem]{Lemma}
\newtheorem{corollary}[theorem]{Corollary}

\theoremstyle{definition}
\newtheorem{definition}[theorem]{Definition}

\newtheorem{example}[theorem]{Example}

\newtheorem{question}[subsection]{Question}

\newcommand\E{\mathbb{E}}
\newcommand\Z{\mathbb{Z}}
\newcommand\R{\mathbb{R}}
\newcommand\Q{\mathbb{Q}}
\newcommand\T{\mathbb{T}}
\newcommand\C{\mathbb{C}}

\newcommand\N{\mathbb{N}}

\newcommand\CB{\mathcal{B}}
\newcommand\CF{\mathcal{F}}

\newcommand\Y{\mathcal{Y}}
\newcommand\CY{\mathcal{Y}}

\newcommand\CS{\mathcal{S}}

\newcommand\ZZ{\mathrm{\bf Z}}
\newcommand\YY{\mathrm{\bf Y}}
\newcommand\XX{\mathrm{\bf X}}

\newcommand\IP{\mathrm{IP}}
\newcommand{\one}{\boldsymbol{1}}
\newcommand\U{\mathsf{U}}
\newcommand\EIP{{\E_{n\in \IP_{\Phi_N}\big((n_j)_{j\in\N}\big)}}}
\newcommand\IPPhi{\IP_{\Phi_N}}

\newcommand{\rat}{\mathrm{rat}}
\newcommand{\Id}{\mathrm{Id}}
\DeclareMathOperator{\sEIP}{\mathbb{E}_{n\in \IP_{\Phi_N}}}

\newcommand{\dip}{\underline{d}_{\IP_\Phi}}
\renewcommand{\d}{\,\mathsf{d}}

\begin{document}


\baselineskip=17pt


\title[Ergodic averages and large intersection property along $\IP$ sets]{Ergodic averages and the large intersection property along $\IP$ sets.}

\author[B. Kra]{Bryna Kra}
\address{Department of Mathematics\\ Northwestern university \\ Evanston, IL 60208 }
\email{kra@math.northwestern.edu}

\author[O. Shalom]{Or Shalom}
\address{Department of Mathematics\\ Bar Ilan University \\ 
Ramat Gan \\
5290002, Israel}
\email{Or.Shalom@math.biu.ac.il}

\date{\today}

\begin{abstract} 
We study multiple ergodic averages along $\IP$ sets, meaning we restrict iterates in the averages to all finite sums of some infinite sequence of natural numbers.  We give  criteria for convergence and divergence in mean of these multiple averages and  derive sufficient conditions for  convergence to the projection onto the space of invariant functions.  For a class of sequences that, roughly speaking, only have rational obstructions to such a limit, we show that the behavior is controlled by nilsystems.  
We also consider pointwise convergence, obtaining convergence and a formula for a set of functions on nilsystems that are dense in $L^2$. Finally, we show that certain correlations have optimally large intersections along an $\IP$ set.
\end{abstract}



\maketitle


\section{Introduction}
\subsection{$\IP$ convergence}
Since Furstenberg's proof~\cite{furstenberg1977ergodic} of Szemer\'edi's Theorem~\cite{szemeredi1975sets} using ergodic theory, dynamical methods have been used to prove numerous generalizations (some examples include~\cite{furstenberg1978ergodic, bergelson1996polynomial,BHK, FK85, furstenberg1991density}). Among the many generalizations, we focus on the $\IP$ version of Szemer\'edi's theorem. An $\IP$, short for {\em infinite parallelepiped} or {\em idempotent} depending on the context, is defined as all the finite sums of some sequence $(n_j)_{j\in\mathbb{N}}$ and an $\IP$-set is a set containing an $\IP$ (for technical reasons, we deviate from the standard definition and allow an $\IP$ to be a multiset). 
Furstenberg and Katznelson~\cite{FK85} showed that for any $\IP$-set, any set of natural numbers with positive upper Banach density contains arbitrary long arithmetic progressions whose common step lies in that $\IP$.  To prove this refinement, they change the type of convergence studied, introducing taking a limit along $\IP$-sets. This type of limit has been used in other contexts to obtain further combinatorial and ergodic generalizations, including for example in~\cite{BergelsonIPSzemeredi}.

In using ergodic methods to prove these results, one is naturally led to study {\em multicorrelation sequences}, meaning sequences of the form 
\begin{equation}
\label{eq:multicor}
a(n) = \int_\XX \prod_{i=0}^k f_i (T^{in}x)\d\mu(x),
\end{equation}
where 
$\XX=(X,\mathcal{B},\mu,T)$ is an invertible measure preserving system and  $f_0,\dots,f_k\in L^\infty(\mu)$ (we defer precise definitions and further background to Section~\ref{sec:background}).
Taking all functions to be the indicator function of some set with positive measure, showing that the $\liminf$ of a multicorrelation sequence is positive corresponds to proving a recurrence statement, and such results have led to generalizations of Szemer\'edi's theorem along subsequences. Convergence in mean has been shown for many multicorrelation sequences,  including~\cite{host2005nonconventional, ziegler2007universal,fw,walsh2012norm}. Further refinements have been given, such as large intersections for the indicator function of a set in a multicorrelation~\cite{BHK,Franpoly,FrantzikinakisKuca2025,MC3} (see also~\cite{ABB,shalom3, shalom2,ABS,Ackelsberg2024} for extensions to abelian groups beyond $\Z$).

For each of these types of results (recurrence, convergence, and large intersections), it is natural to ask if there are $\IP$ versions, and this is our focus.  
We start in Section~\ref{sec:conv} by giving necessary and sufficient conditions for mean convergence of an average along an $\IP$.  Using spectral criteria, in Theorem~\ref{criteria}, we show when such an average for a function converges to its projection onto the invariant $\sigma$-algebra.  This allows us to deduce sufficient conditions for convergence along an $\IP$ in Section~\ref{sec:conv-more} and an example of divergence along an $\IP$ in Section~\ref{sec:diverge}. 
In particular, we show that for $\IP$ sequences with rational spectrum (see Definition~\ref{spectrum}) we can describe the limit (Theorem~\ref{METrational}), and this is the class of $\IP$s for which we can describe more detailed behavior (Theorem~\ref{formula}).

In Section~\ref{sec:multiple}, we turn our attention to multiple ergodic averages along an $\IP$. Proving convergence for these multiple averages requires the development of $\IP$ versions of the tools used for linear averages, including a version of the van der Corput Lemma (Lemma~\ref{vdc}), a version of Gowers-Host-Kra seminorms (Definition~\ref{def:seminorms}), and $\IP$ cubic measures (Definition~\ref{def:cubic-measure}). 
This allows us to show (Theorem~\ref{HKcharacteristic}) that  for sequences with rational spectrum, the behavior of the multiple ergodic average along an $\IP$ is controlled by the same factors that control the well understood case of linear averages.  

We then use the tools developed to derive several applications.  In Theorem~\ref{formula}, we give a limit formula for certain multiple averages along $\IP$s, and this in turn is used to derive a large intersection property for progressions of length three and four along an $\IP$ with rational spectrum.  The ergodic version of these results are given in Theorems~\ref{three} and~\ref{four}, and via variant of Furstenberg's correspondence principle for ergodic systems (see~\cite[Proposition 3.1]{BHK}), we immediately deduce an application in the integers (see Definition~\ref{IP-density:def} for the notion of $\IP$ density). 
To state the result, recall that if $E\subseteq \mathbb{Z}$, we define the {\em upper Banach density} $d^*(E)$ by 
$$d^*(E)=\lim_{N\rightarrow\infty}\sup_{M\in\Z} \frac{|E\cap \{M+1,\dots,M+N\}|}{N}.$$
\begin{theorem}
\label{th:combinatorial}
Let $(n_j)_{j\in\mathbb{N}}$ be a sequence of natural numbers with rational spectrum and let  $E\subseteq\mathbb{Z}$ be a set with $d^*(E)>0$. Then for all coprime integers $\ell_1,\ell_2\in \mathbb{Z}$ and all $\varepsilon>0$, 
we have that both 
$$\dip\left\{n\in \IP\bigl((n_j)_{j\in\mathbb{N}}\bigr) : d^*\bigl(E\cap (E-\ell_1n)\cap (E-\ell_2n)\bigr) > d^*(E)^3 -\varepsilon\right\}$$
and
$$\dip\left\{n\in \IP\bigl((n_j)_{j\in\mathbb{N}}\bigr) : d^*\bigl(E\cap 
(E-\ell_1n) \cap (E-\ell_2n) \cap (E-(\ell_1+\ell_2)n)\bigr)>d^*(E)^4-\varepsilon\right\}$$ 
are positive for any increasing F\o lner sequence $\Phi=(\Phi_N)_{N\in\mathbb{N}}$.  
\end{theorem}

As a sample application (see Exmaple~\ref{example:10}), this applies to the $\IP$ generated by $1,10,100,\dots$, which generates the set of all numbers in base $10$ that can be written only using the digits $0$ and $1$.

\subsection{Open questions}
It is natural to ask if the sets of large intersections in Theorem~\ref{th:combinatorial} are large with respect to other measurements of largeness. One phrasing of this question is the following syndetic version of this question.
\begin{question}
    Let $(n_j)_{j\in\mathbb{N}}$ be a sequence of natural numbers with rational spectrum, let $E\subseteq \mathbb{Z}$ be a set with positive upper Banach density $d^*(E)>0$, and let $\varepsilon>0$. Do finitely many translates of the set
    $$\{n\in \IP\big((n_j)_{j\in\mathbb{N}}\big) : d^*\big(E\cap (E-n)\cap (E-2n)\big)>d^*(E)^3-\varepsilon\}$$ 
    cover $\IP\big((n_j)_{j\in\mathbb{N}}\big)?$ 
    Does the analogous statement hold for four term arithmetic progressions? 
\end{question}

From the example in Theorem~\ref{false}, we know that there exist $\IP$s for which the set of large intersections is trivial, and this leads to the next question.
\begin{question}
    Are there $\IP$s other than those with rational spectrum for which the set of large intersections in Theorem~\ref{th:combinatorial} for three (or four) sets is nonempty? 
\end{question} 

A related question on these intersections is if there is some other bound for which the large intersection property holds for all $\IP$s.  It follows from Furstenberg and Katznelson~\cite{FK85} that there exists a constant $C>0$ such that 
    \begin{equation}
        \label{eq:lowerC}
\{n\in \IP\big((n_j)_{j\in\mathbb{N}}\big) : d^*\big(E\cap (E-n)\cap (E-2n)\big) > C\}
 \end{equation}
 is nonempty, and it further follows from~\cite{bergelson2000aspects} that $C=C(d^*(E))$ is a constant depending only on $d^*(E)$.  This leads to the next question. 
\begin{question}
    Let $E\subseteq\mathbb{N}$ with $d^*(E)=\delta$ for some $\delta>0$.  What is the optimal value of $C = C(\delta)$ in~\eqref{eq:lowerC}? 
\end{question}
We note that we phrase this as the optimal value, as in Theorem~\ref{false} we show  that $C(\delta)\leq \delta^{-c\log \delta}$ for some absolute constant $c>0$.

A last question that we pose is on $\IP$ pointwise convergence. 
\begin{question}
    Let $\XX=(X,\mathcal{B},\mu,T)$ be an ergodic invertible measure preserving system,  let $(n_j)_{j\in\mathbb{N}}$ be a sequence with rational spectrum, and let $f\in L^\infty(\mu)$. Does the limit     $$\lim_{N\rightarrow\infty} \EIP T^n f(x)$$ exist for $\mu$-almost every  $x\in X$?
\end{question}
While such a result is of interest on its own,  it would also simplify the computations of the limit of an average along an $\IP$ in nilsystems, leading to simplifications of the proofs in Section~\ref{limit:sec}.

\subsection*{Acknowledgements}
The first author was partially supported by the National Science Foundation grant DMS-2348315. The second author was partially supported by the National Science Foundation grant DMS-1926686. We thank Ethan Ackelsberg for useful discussions during the preparation of this article and for pointing out useful references.

\section{Background and adaptations for $\IP$ sequences}
\label{sec:background}
\subsection{Measure preserving systems}
An \emph{invertible measure preserving system} $\XX=(X,\mathcal{B},\mu,T)$ is a quadruple in which  $(X,\mathcal{B},\mu)$ is a probability space and $T\colon X\rightarrow X$ is an invertible, measurable, measure preserving transformation (meaning that $\mu(A)=\mu(T^{-1}(A))$ for all measurable $A\subseteq X$). When the context is clear, we shorten this and refer to an invertible measure preserving system as a \emph{system}. 
Throughout, we assume that all systems $\XX=(X,\CB,\mu,T)$ are regular, meaning that $X$ is a compact metric space, $\CB$ is the Borel-$\sigma$ algebra, and $\mu$ is a regular measure.

Given a system $\XX=(X,\mathcal{B},\mu,T)$,  let $L^2(\mu)$ denote the Hilbert space of all complex valued square integrable functions on $X$ modulo $\mu$-almost everywhere equivalence. We make the usual abuse of notation, using $T$ both for the transformation on $X$ and for the the associated  unitary operator $T\colon L^2(\mu)\rightarrow L^2(\mu)$ defined by $f\mapsto f\circ T$ for $f\in L^2(\mu)$. The measure preserving system is \emph{ergodic} if the only $T$-invariant functions are constant $\mu$-almost everywhere.

Let $\CS^1$ denote the one dimensional circle.  For $t\in \CS^1$, the \emph{$t$-eigenspace} is the closed Hilbert space consisting of all $f\in L^2(\mu)$ satisfying $Tf=t\cdot f$, and any nonzero function $f\in L^2(\mu)$ in the $t$-eigenspace is an \emph{eigenfunction with eigenvalue $t$}. For any $t\in \CS^1$, we let $P_t\colon L^2(\mu)\rightarrow L^2(\mu)$ denote orthogonal projection onto the $t$-eigenspace.
\subsection{Factors and cocycles}
If $\XX=(X,\mathcal{B},\mu,T)$ and $\YY=(Y,\mathcal{D},\nu,S)$ are  systems, 
we say that $\YY$ is a \emph{factor} of $\XX$ (or equivalently $\XX$ is  an \emph{extension} of $\YY$)  if there exists a measurable map $\pi\colon X\to Y$ such that $\pi\mu = \nu$ and $S\circ\pi = \pi\circ T$ holds $\mu$-almost everywhere. When the map $\pi$ is also invertible with a measurable inverse, we say that the systems $\XX$ and $\YY$ are \emph{isomorphic}.

If $\YY=(Y,\CY,\nu,S)$ is a factor of the system  $\XX=(X,\mathcal{B},\mu,T)$ with factor map $\pi\colon X\to Y$ and $f\in L^2(\mu)$, let $\E(f\mid\YY)$ denote the {\em  conditional expectation} of $f$ onto $Y$, defined as the unique function in $L^2(\nu)$
such that $\int_{\pi^{-1}(A)}f\d\mu = \int_A \E(f\mid\YY)\d\nu$ for all measurable set $A\subset\CY$. To avoid confusion, we use $\E(f\mid \Y)$ to denote the lift of $\E(f\mid \YY) $ to $L^2(\mu)$.

Let $\XX=(X,\mathcal{B}_X,\mu_X,T)$ be a system and $K=(K,\mathcal{B}_K,\mu_K)$ be a compact abelian group equipped with the Borel $\sigma$-algebra and Haar measure. A \emph{cocycle} on $\XX$ with values in $K$ is a measurable map $\rho\colon X\rightarrow K$. The \emph{extension defined by $\rho$} is the system $\XX\times_\rho K = (X\times K,\mathcal{B}_X\times \mathcal{B}_K,\mu_X\times\mu_K,T_\rho)$, where 
$$T_\rho (x,u) = (Tx,\rho(x)+u).$$ A cocycle $\rho\colon X\rightarrow K$ gives rise to an extended map $\rho\colon \mathbb{Z}\times X\rightarrow K$ defined by ($\Id_K$ denotes the identity in $K$)
$$\rho(n,x)=\begin{cases} \sum_{i=0}^{n-1}\rho(T^ix) & n\geq 1\\
\Id_K & n=0\\
-\sum_{i=n+1}^{0}\rho(T^{i}x)& n\leq -1
\end{cases}$$
and 
$$T_\rho^n (x,u) = (T^nx,\rho(n,x)+u).$$

Two cocycles $\rho,\rho'\colon X\rightarrow K$ are  \emph{cohomologous} if there exists a measurable map $F\colon X\rightarrow K$ such that $\rho(x)=\rho'(x)+F(Tx)-F(x).$ We note that if $\rho,\rho'$ are cohomologous then the systems $\XX\times_\rho K$ and $\XX\times_{\rho'} K$ are isomorphic and the isomorphism is given by $(x,t)\mapsto (x,t-F(x))$.

\begin{definition}[Continuous vertical characters]
    Let $\XX=(X,\mathcal{B}_X,\mu_X,T)$ be a system, let $K$ be a compact abelian group, and let $\rho\colon X\rightarrow K$ be a cocycle. A function $f\colon X\times_\rho K\rightarrow \C$ is a \emph{continuous vertical character} if there exists a character $\chi\colon K\rightarrow \CS^1$ and a continuous function $g\colon X\rightarrow \mathbb{C}$ such  that $f(x,k) = g(x)\cdot\chi(k).$
\end{definition}

We note for use in the sequel that it follows from  the Stone-Weierstrass theorem that the space spanned by the continuous vertical characters is dense in the continuous functions $C(X\times_\rho K)$.
\begin{lemma}\label{dense}
     Let $\XX=(X,\mathcal{B}_X,\mu_X,T)$ be a system, let $K$ be a compact abelian group and $\rho\colon X\rightarrow K$ a cocycle. The space spanned by the continuous vertical characters on $X\times_\rho K$ is dense in $C(X\times_\rho K)$.
\end{lemma}

\subsection{Definition of an $\IP$}
\begin{definition}
If $(n_j)_{j\in\mathbb{N}}$ is a sequence of natural numbers,  define 
$\IP\bigl((n_j)_{j\in\mathbb{N}}\bigr)$ to be the multiset  containing all finite sums of the sequence $(n_j)_{j\in\mathbb{N}}$, meaning   
\[
\IP\bigl((n_j)_{j\in\mathbb{N}}\bigr) = \{n_{i_1}+\dots+n_{i_\ell} : i_1,\dots,i_\ell \text{ are distinct} \}, 
\]
with the convention that the empty sum $0$ is included.  We refer to the numbers  $(n_j)_{j\in\N}$ as  the {\em generators}, and when the generators are clear from the context we refer to this simply as an {\em $\IP$}.  
\end{definition}

Our definition differs slightly from some uses in the literature, where there is no assumption that $0$ lies in the $\IP$
and an $\IP$ is defined as a set rather,  than a multiset. 
The assumption that an $\IP$ contains $0$ is merely for notational convenience, but the use of a multiset is necessary for our analysis. We could avoid this by taking the generators to be a  sequence $(n_j)_{j\in\mathbb{N}}$ that grows sufficiently fast such that $n_j>\sum_{i=1}^{j-1}n_i$ for all $j\in\N$, but we prefer to work in the more general setting.  
We note that in the literature, a set is sometimes called an $\IP$-set if it contains an $\IP$.  However, as we also prove convergence results, we can not adopt that convention. 

\begin{example}
\label{example:10}
 Set $n_j = 10^{j-1}$ for all $j\geq 1$. Then $\IP\bigl((n_j)_{j\in\mathbb{N}}\bigr)$ is the set of all natural numbers that can be written in base $10$ only using the digits $0$ and $1$, meaning that $$\IP\bigl((n_j)_{j\in\mathbb{N}}\bigr) = \{0,1,10,11,100,101,110,111,1000,\dots\}.$$
 Note that this $\IP$ is a set, not a multiset, as all finite sums are distinct.  
  If instead we take the generators  $n_j=2^{j-1}$ for $j\geq 1$, then  $\IP\bigl((n_j)_{j\in\mathbb{N}}\bigr)=\mathbb{N}$.
    \end{example}
    
   It is easy to see that there are many subsets of integers that are not $\IP$s.  For example, it is easy to check that an $\IP$ contains infinitely many multiples of a given positive integer.

\subsection{Averages along an $\IP$}

A {\em F\o lner sequence $\Phi=(\Phi_N)_{N\in\N}$ in  $\N$} is a sequence of finite subsets of natural numbers satisfying 
\begin{equation}\label{Folner}
\lim_{N\rightarrow\infty} \frac{|(a+\Phi_N )\Delta \Phi_N|}{|\Phi_N|} =0
\end{equation}
for all $a\in\mathbb{N}$, where $\Delta$ denotes the symmetric difference. A standard example of a F\o lner sequence is taking intervals, such as $\Phi_N = [1,N]$ or  $\Phi_N = [N^2,N^2+N]$. We say that a F\o lner sequence $\Phi_N$ is \emph{increasing} if $\Phi_N=[M,a_N]$ for some fixed number $M$ and a sequence $a_N$ tending to infinity as $N\rightarrow\infty$.

To define averages along an $\IP$, we adapt  F\o lner sequences to our context.  
\begin{definition}
    Let $(n_j)_{j\in\mathbb{N}}$ be a sequence of natural numbers and $\Phi=(\Phi_N)_{N\in\mathbb{N}}$ be a F\o lner sequence in $\N$. The \emph{$\IP$-F\o lner sequence associated to $(\Phi_N)_{N\in\mathbb{N}}$} is the sequence    $$\IP_{\Phi_N}\bigl((n_j)_{j\in\mathbb{N}}\bigr) = \{n_{i_1}+\dots+n_{i_\ell} : i_1,\dots,i_\ell\in \Phi_N \text{ are distinct} \}.$$
    If the F\o lner sequence $(\Phi_N)_{N\in\mathbb{N}}$ is increasing, we say that the  $\IP$-F\o lner sequence $\IP_{\Phi_N}\bigl((n_j)_{j\in\mathbb{N}}\bigr)$ is \emph{increasing}.
    When the generating sequence $(n_j)_{j\in\N}$ is clear from the context, we omit it from the notation. 
\end{definition}
Given a F\o lner sequence $\Phi=(\Phi_N)_{N\in\mathbb{N}}$ in $\N$ and sequence $(a_n)_{n\in\N}$, we use the shorthand notation 
\[\EIP a_n  = \frac{1}{|\IP_{\Phi_N}\bigl((n_j)_{j\in\mathbb{N}}\bigr)|}\sum_{n\in\IP_{\Phi_N}\bigl((n_j)_{j\in\mathbb{N}}\bigr)}a_n, 
\]
and when the generators are clear from the context, we write $\sEIP a_n$.

This allows us to  define ergodic averages along $\IP$s.

\begin{definition}[The ergodic average]
Let $\XX=(X,\mathcal{B},\mu,T)$ be an invertible measure preserving system, let $(n_j)_{j\in\mathbb{N}}$ be a sequence of natural numbers,   let $(\Phi_N)_{N\in\mathbb{N}}$ be a F\o lner sequence, and let $f\in L^2(\mu)$.  The {\em ergodic average of  $f$ along the $\IP$-F\o lner sequence $\IP_{\Phi_N}\bigl((n_j)_{j\in\mathbb{N}}\bigr)$} is defined to be the limit 
\begin{equation}
\label{eq:ergodic-average}
\lim_{N\rightarrow\infty}\E_{n\in \IP_{\Phi_N}\big((n_j)_{j\in\mathbb{N}}\big)} T^n f
\end{equation}
taken in $L^2(\mu)$.
\end{definition}

Considering the ergodic average along $\IP_{[1,N]}\big((2^{j-1})_{j\in\mathbb{N}}\big)$, von Neumann's mean ergodic theorem states that this average exists for all $f\in L^2(\mu)$ and converges to the projection onto the $T$-invariant functions. In contrast, we give examples of $\IP$-F\o lner sequences where the ergodic averages do not converge, and examples  where the averages converge but the limit is not  projection onto the $T$-invariant functions.  

\section{Ergodic averages along an $\IP$-sequence}
\label{sec:conv}

\subsection{Convergence along an $\IP$-F\o lner sequence}

In this section, we give necessary and sufficient conditions in Theorem~\ref{criteria} for mean convergence of ergodic averages along an $\IP$-F\o lner sequence, with criteria for when the limit is the projection onto the $T$-invariant functions.   We start by recalling some standard definitions and setting notation. 

Let $\delta_1$ denote the \emph{Dirac function} on the  circle $\CS^1$, meaning 
\begin{equation}\label{delta}\delta_1(\alpha) = \begin{cases} 1 & \alpha = 1\\ 0 & \alpha \neq 1,\end{cases}
\end{equation}
 Our first condition for convergence is a consequence of the spectral theorem for unitary operators. 

\begin{theorem}
\label{spectralcriteria}
       Let $(n_j)_{j\in\mathbb{N}}$ be a sequence of natural numbers and let $\Phi=(\Phi_N)_{N\in\N}$ be a F\o lner sequence. The following are equivalent: 
       \begin{enumerate}
       \item 
       \label{eq:oneone}
       For all invertible measure preserving systems $\XX=(X,\mathcal{B},\mu,T)$ and all $f\in L^2(\mu)$, the ergodic averages~\eqref{eq:ergodic-average} along  $\IP_{\Phi_N}\bigl((n_j)_{j\in \mathbb{N}}\bigr)$ exists,
       \item 
       \label{eq:twotwo}
       The limit 
       \begin{equation}
\label{spectralcriteriaeq}
\lim_{N\rightarrow\infty}\E_{n\in \IP_{\Phi_N}\big((n_j)_{j\in \mathbb{N}}\big)} \alpha^n   
    \end{equation}
    exists for all  $\alpha\in \CS^1$.
       \end{enumerate} Furthermore, the following are equivalent: 
        \begin{enumerate}[resume]
            \item 
            \label{eq:threethree}
              For all invertible measure preserving systems $\XX=(X,\mathcal{B},\mu,T)$ and all $f\in L^2(\mu)$, the ergodic averages~\eqref{eq:ergodic-average} along  $\IP_{\Phi_N}\bigl((n_j)_{j\in \mathbb{N}}\bigr)$  converges to the projection onto the $T$-invariant functions.
            \item 
            \label{eq:fourfour}
            The limit in~\eqref{spectralcriteriaeq} is equal to $\delta_1(\alpha)$.
        \end{enumerate}  
\end{theorem}

\begin{proof}
We start by proving the equivalence of~\eqref{eq:oneone} and~\eqref{eq:twotwo}. First suppose that for every invertible measure preserving system $\XX=(X,\mathcal{B},\mu,T)$ and all $f\in L^2(\mu)$, the ergodic averages along $\IP_{\Phi_N}\bigl((n_j)_{j\in \mathbb{N}}\bigr)$ exist. 
    Given $\alpha\in \CS^1$,  take $\XX$ to be the ergodic rotation on $\CS^1$ by $\alpha$ and let $f\colon X\rightarrow \CS^1$ be the natural embedding. Then,
   \begin{equation}\label{averagef}\sEIP T^n f =\left(\sEIP \alpha^n\right) \cdot f
   \end{equation} and it follows immediately that the limit~\eqref{spectralcriteriaeq} exists. Conversely, suppose that ~\eqref{spectralcriteriaeq} exists for all $\alpha\in \CS^1$ and let $d\rho$ denote the projection valued measure for $T$, meaning that 
    $$T^n = \int_{\CS^1} t^n \d\rho(t)$$ for all $n\in \N$. Then by the Lebesgue dominated convergence theorem, 
\begin{equation}
    \label{eq:phi}
    \lim_{N\rightarrow\infty}\sEIP T^n= \int_{\CS^1} \lim_{N\rightarrow\infty}\sEIP t^n \d\rho(t) = \int_{\CS^1} \phi(t)\d\rho(t),
    \end{equation}
 where $\phi(t) = \lim_{N\rightarrow\infty}\sEIP  t^n$ is the pointwise limit which exists by assumption.
 In particular, we see that the ergodic average converges for all $f\in L^2(\mu)$.

 To check the equivalence of~\eqref{eq:threethree} and~\eqref{eq:fourfour}, note that iff $\alpha=1$, it is immediate that the limit~\eqref{spectralcriteriaeq} equals $1$. Supposing that~\eqref{eq:threethree} holds, assume that $\alpha\not =1$ and again take $\XX$ to be the ergodic rotation by $\alpha$ and $f\colon X\to\CS^1$ to be the natural embedding.  Since $\int_X f\d\mu =0$, it follows from~\eqref{averagef} that the limit ~\eqref{spectralcriteriaeq} is $0$. Conversely, suppose that~\eqref{eq:fourfour} holds, meaning  that the map $\phi$ in~\eqref{eq:phi} satisfies $\phi(t) = \delta_1(t)$. Using the same argument as in the first part, the ergodic average converges to the projection $\rho(\{1\})$. Observe that
 $$T\rho(\{1\}) = \int_{S^1} t\cdot \delta_1(t) ~d\rho(t) = \int_{S^1} \delta_1(t)~d\rho(t)$$
 and similarly $\rho(\{1\})T=\rho(\{1\})$. It follows that $\rho(\{1\})$ is the projection onto the subspace of $T$-invariant functions.
\end{proof}

To refine this characterization of convergence, we use the next result to rewrite averages as infinite products. This use also justifies our conventions on the definition of an $\IP$-set,  defining it as a multiset and always including $0$ in the set. 
\begin{proposition}\label{avgtoprod}
    Let $(n_j)_{j\in\mathbb{N}}$ be a sequence of natural numbers, let $\Phi=(\Phi_N)_{N\in\mathbb{N}}$ be a F\o lner sequence in $\N$, let $\alpha\in \CS^1$, and set  $z_j = \frac{1+\alpha^{n_j}}{2}$ for  $j\in\N$.  
    For every $N\in\mathbb{N}$, we have 
    $$\E_{n\in \IP_{\Phi_N}\bigl((n_j)_{j\in\mathbb{N}}\bigr)} \alpha^n = \prod_{j\in \Phi_N} z_j.$$
\end{proposition}

\begin{proof}
    We have
    $$\prod_{j\in \Phi_N} z_j = \prod_{j\in\Phi_N} \left(\frac{1+\alpha^{n_j}}{2}\right) = \frac{\prod_{j\in \Phi_N}(1+\alpha^{n_j})}{2^{|\Phi_N|}}$$
    and by expanding the terms, we obtain
     \[\prod_{j\in \Phi_N} z_j = \frac{\sum_{n\in\IP_{\Phi_N}\bigl((n_j)_{j\in\mathbb{N}}\bigr)}\alpha^n}{2^{|\Phi_N|}} = \sEIP \alpha^n. \quad\qedhere
  \]
\end{proof}

Let  $\arg\colon \CS^1\rightarrow \mathbb{R}/2\pi \mathbb{Z}$ denote the inverse of the isomorphism $t\mapsto e^{it}$. We have the following criteria for convergence.
\begin{theorem}[Criteria for convergence]\label{criteria}
Let $(n_j)_{j\in\mathbb{N}}$ be a sequence of natural numbers and let $\Phi=(\Phi_N)_{N\in\mathbb{N}}$ be a F\o lner sequence in $\N$.
\begin{enumerate}
  \item 
  \label{item:one} 
  For all invertible measure preserving system $\XX=(X,\mathcal{B},\mu,T)$, the  ergodic average~\eqref{eq:ergodic-average} along $\IP_{\Phi_N}\bigl((n_j)_{j\in\mathbb{N}}\bigr)$ converges to the projection onto the $T$-invariant functions if and only if 
  $$\lim_{N\rightarrow\infty}\prod_{j\in \Phi_N} \frac{(1+\cos(\arg(\alpha^{n_j}))}{2} = \delta_1(\alpha) 
  $$
 for all  $\alpha\in \CS^1$. 

  \item 
  \label{item:two}
   For all invertible measure preserving system $\XX=(X,\mathcal{B},\mu,T)$, the  ergodic average~\eqref{eq:ergodic-average} along $\IP_{\Phi_N}\bigl((n_j)_{j\in\mathbb{N}}\bigr)$ converges
  if and only if each $\alpha\in \CS^1$ satisfies either $$\lim_{N\rightarrow\infty}\prod_{j\in \Phi_N} \frac{(1+\cos(\arg(\alpha^{n_j}))}{2}=0$$ 
  or $$\lim_{N\rightarrow\infty} \sum_{j\in \Phi_N} \arctan\left(\frac{\sin(\arg(\alpha^{n_j}))}{1+\cos(\arg(\alpha^{n_j}))}\right)$$ exists.
\end{enumerate}
\end{theorem}
  
\begin{proof}
    Fixing  $j\in\mathbb{N}$ and $\alpha\in \CS^1$, set $z_j = z_j(\alpha) = \frac{1+\alpha^{n_j}}{2}$. 
    Rewriting, we have that $z_j = r_j e(i\theta_j)$ for some $r_j\geq 0$ and $\theta_j\in (-\frac{\pi}{2},\frac{\pi}{2})$. Using the identity
    \[z_j(\alpha) = \frac{1+\cos(\arg(\alpha^{n_j}))}{2} + i\frac{\sin(\arg(\alpha^{n_j}))}{2},\]
we can rewrite 
$r_j=\sqrt{\frac{1+\cos(\arg(\alpha^{n_j}))}{2}}$ and $\theta_j = \arctan\left(\frac{\sin(\arg(\alpha^{n_j}))}{1+\cos(\arg(\alpha^{n_j}))}\right)$. In particular, it follows from Proposition~\ref{avgtoprod}  that 
    $$\E_{n\in \IP_{\Phi_N}
    }    
    \alpha^n = \Bigl(\prod_{j\in\Phi_N}r_j\Bigr)\cdot e(i\sum_{j\in\Phi_N}\theta_j).$$
    By Theorem~\ref{spectralcriteria}, the ergodic average along $\IP_{\Phi_N}\bigl((n_j)_{j\in\mathbb{N}}\bigr)$ exists if and only if this quantity  converges as $N\rightarrow\infty$. 
    We use the general criterion that a  sequence of complex numbers converges if and only if either the norms converge to zero or the norms converge and the arguments converge. 
    In the first case, if the norms converge to zero for all $\alpha\not = 1$, then $\E_{n\in\IP_{\Phi_N}} \alpha^n = \delta_1(\alpha)$ and so Part~\eqref{item:one} follows from  Theorem~\ref{spectralcriteria}. Otherwise, since $z_j$ is the average of two numbers in the unit circle, it follows that $|z_j|\leq 1$ and so the norm of $\E_{n\in\IP_{\Phi_N}}\alpha^n$ always converges. 
    If the limit is nonzero, then by Theorem~\ref{spectralcriteria} the limit exists if and only if the arguments converge, proving Part~\eqref{item:two}.
\end{proof}

\subsection{Sufficient conditions for convergence}
\label{sec:conv-more}
We start with a simple proposition showing that for the average along an $\IP$ to converge to the projection onto the $T$-invariant functions, it suffices that the average along the generators does so.  
\begin{proposition}
    Let $(n_j)_{j\in\mathbb{N}}$ be a sequence of natural numbers and let $\Phi=(\Phi_N)_{N\in\mathbb{N}}$ be a F\o lner sequence in $\N$. Assume that for all ergodic rotations $\ZZ=(Z,\mathcal{C},\nu,R_a)$ and all $f\in L^2(\nu)$, we have 
    \[\lim_{N\rightarrow\infty} \E_{n\in \Phi_N} f(a^nx) = \int_Zf\d\nu\] 
    in $L^2(\nu)$.   
    Then for every ergodic system $\XX=(X,\mathcal{B},\mu,T)$ and all $f\in L^2(\mu)$, we have
    $$\lim_{N\rightarrow\infty} \EIP T^n f = \int_\XX f \d\mu.$$
\end{proposition}
\begin{proof}
    The hypothesis of convergence in the proposition holds if and only if $\lim_{N\rightarrow\infty} \E_{j\in \Phi_N} t^{n_j} = \delta_1(t)$ for all $t\in \CS^1$. Fix $t\not =1$, and consider
    $$A_N := \{ j\in \Phi_N :  |t^{n_j}-1| < 1/10\}.$$ 
Using the triangle inequality, we have that 
    \[\Big|\frac{1}{|\Phi_N|} \sum_{j\in A_N} t^{n_j} \Big| \leq \Big|\E_{j\in \Phi_N} t^{n_j}\Big| + \Big|\frac{1}{|\Phi_N|}\sum_{j\not\in A_N} t^{n_j}\Big| .
    \]
    Since $|t^{n_j}-1|<1/10$ for all $j\in A_N$, the left hand side is bounded from  below by $\frac{9|A_N|}{10|\Phi_N|}$. Fixing arbitrary  $\varepsilon>0$, it follows that for sufficiently large $N$ we have that 
    \[\frac{9|A_N|}{10|\Phi_N|} \leq \varepsilon + \Bigl(1-\frac{|A_N|}{|\Phi_N|}\Bigr).\] 
    Choosing any $0 < \varepsilon <9/10$, it follows that $\limsup_{N\rightarrow\infty} \frac{|A_N|}{|\Phi_N|}<1$, 
    and so for all sufficiently large $N$, we have $\frac{|A_N|}{|\Phi_N|}<1-c$ for some sufficiently small $c>0$. In particular, the number of $j\in \Phi_N$ satisfying $|t^{n_j}-1|>1/10$ is arbitrarily large as $N\rightarrow\infty$. By Proposition~\ref{avgtoprod}, it follows that $$\lim_{N\rightarrow\infty}\sEIP t^n = \lim_{N\rightarrow\infty} \prod_{j\in \Phi_N} \Bigl(\frac{1+t^{n_j}}{2}\Bigr)=0.$$ 
    As this holds for all $1\not=t\in \CS^1$, convergence follows from Theorem~\ref{spectralcriteria}.
\end{proof}

  This proposition  applies to many  sequences, including polynomial sequences (see for example~\cite{walsh2012norm}), the primes and  polynomials in primes~\cite{wooley, FHK}, and many other sequences such as those in~\cite{Boshernitzan2005sequences, Fran2010}.

A particular class of dynamical systems in which we can describe convergence is when the eigenspaces generate $L^2(\mu)$ and additionally all nontrivial eigenspaces are associated to rational eigenvalues.
\begin{definition}
    If $\XX=(X,\mathcal{B},\mu,T)$ is an invertible measure preserving system,  let $\mathcal{K}_{\rat}(\XX)$ denote the minimal factor of $\XX$ such that any eigenfunction of $\XX$ with rational eigenvalue is measurable with respect to $\mathcal{K}_{\rat}(\XX)$.  The system $\XX$ is \emph{rational} if $\XX=\mathcal{K}_{\rat}(\XX)$.
\end{definition}
Thus $\mathcal{K}_{\rat}$ is the factor generated by the eigenfunctions that are associated  to the rational eigenvalues.  It is immediate that any finite system is rational. 
A more interesting example is obtained by taking $X = \mathbb{Z}_p$ is the $p$-adic numbers for some fixed prime $p$, equipped with the Borel $\sigma$-algebra, Haar measure, and the transformation is translation by some fixed $t\in \mathbb{Z}_p$. Any character $\chi$ of $\mathbb{Z}_p$ is an eigenfunction whose eigenvalue $\chi(t)$ is a $p^m$ root of unity for some $m\in\mathbb{N}$, and these characters form an orthonormal basis for $\mathbb{Z}_p$. Note that if the group generated by $t$ is dense, then the  system is also ergodic.

Recall that a function $f\colon X\to \C$ is \emph{$1$-bounded} if $|f(x)|\leq 1$ for all $x\in X$ and 
we  let $P_t\colon L^2(\mu)\rightarrow L^2(\mu)$ denote the projection onto the $t$-eigenspace of $T$.

\begin{theorem}[Convergence in  $\mathcal{K}_\rat(\XX)$]\label{METKrat}
Let $(n_j)_{j\in \mathbb{N}}$ be a sequence of natural numbers and let $\Phi=(\Phi_N)_{N\in\mathbb{N}}$ be an increasing F\o lner sequence. Then there exists a $1$-bounded function $\omega_{\Phi}\colon e^{2\pi i \mathbb{Q}}\rightarrow \mathbb{C}$ such  that for every invertible measure preserving system $\XX=(X,\mathcal{B},\mu,T)$ and $f\in L^2(\mu)$ measurable with respect to $\mathcal{K}_{\rat}(\XX)$, we have
$$\lim_{N\rightarrow\infty}\EIP T^n f = \sum_{t\in e^{2\pi i \mathbb{Q}}} \omega_{\Phi}(t)\cdot P_t(f)$$ in $L^2(\mu)$.  
\end{theorem}
\begin{proof}
  Let $(n_j)_{j\in\mathbb{N}}$ and $\Phi=(\Phi_N)_{N\in\mathbb{N}}$ be as in the statement. We first check that the limit 
    \begin{equation}
\label{eq:omega-limit}
  \omega_{\Phi}(t):=\lim_{N\rightarrow\infty}\sEIP t^n
    \end{equation}
    exists for all $t\in e^{2\pi i\mathbb{Q}}$. By Proposition~\ref{avgtoprod}, this limit is equal to
    $$\lim_{N\rightarrow\infty}\prod_{j\in\Phi_N} \left(\frac{1+t^{n_j}}{2}\right).$$
    We consider two cases.  If $t^{n_j}=1$ for all but finitely many values of $j$, 
    then this infinite product is eventually equal to $$\prod_{\{j\text{ } :\text{ } t^{n_j}\not = 1\}} \left(\frac{1+t^{n_j}}{2}\right) $$ 
    and so the limit exists. Otherwise, there are infinitely many $j$ such that 
    $t^{n_j}\not= 1$. Then we can choose some $\varepsilon_0>0$, depending only on $t$, such that $\left|\left(\frac{1+t^{n_j}}{2}\right)\right|<1-\varepsilon_0$. Since the F\o lner sequence $\Phi$  is increasing, it eventually contains arbitrarily many such  $j$ and \[\lim_{N\rightarrow\infty}\prod_{j\in\Phi_N} \left(\frac{1+t^{n_j}}{2}\right)=0.\]
    Thus the limit~\eqref{eq:omega-limit} exists.

  Since $f$ is measurable with respect to $\mathcal{K}_{\rat}(
    \XX)$, we can write $$f= \sum_{t\in e^{2\pi i \mathbb{Q}}} P_t(f).
    $$
    Since $T (P_t(f)) = t\cdot P_t(f)$, we conclude that
    $$\lim_{N\rightarrow\infty} \EIP T^n f = \sum_{t\in e^{2\pi i \mathbb{Q}}} \omega_{\Phi}(t) P_t(f),$$ and the statement follows.
\end{proof}
This leads to the following definition.
\begin{definition}[Spectrum of a sequence]\label{spectrum}
     The \emph{spectrum} $\sigma((n_j)_{k\in\mathbb{N}})$ of a sequence of natural numbers $(n_j)_{j\in\mathbb{N}}$ is defined to be the group generated by the complement of the set
    $$\bigcap_{m\in\mathbb{N}} \left\{\alpha \in \CS^1 : \lim_{d\rightarrow\infty} \prod_{j=m}^d \left(\frac{1+\alpha^{\ n_j}}{2}\right)=0\right\}.$$ We say that a sequence has \emph{rational spectrum} if $\sigma((n_j)_{j\in\mathbb{N}})\subseteq e^{2\pi i \mathbb{Q}}.$ 
\end{definition}
The intersection over $m$ is taken to ensure that for every $\alpha\not\in \sigma((n_j)_{j\in\mathbb{N}})$, and 
every increasing F\o lner sequence $\Phi=(\Phi_N)_{N\in\mathbb{N}}$, the quantity
    \begin{equation}\label{omega}
        \omega_{\Phi}(\alpha):=\lim_{N\rightarrow\infty} \prod_{j\in \Phi_N} \left(\frac{1+\alpha^{n_j}}{2}\right)
    \end{equation}
    is zero.  

The next result follows quickly from  Theorem~\ref{METKrat}. 
\begin{corollary}\label{METrational}
    Let $\XX=(X,\mathcal{B},\mu,T)$ be an invertible measure preserving system and let $(n_j)_{j\in\mathbb{N}}$ be a sequence with rational spectrum. Then for any increasing F\o lner sequence $\Phi=(\Phi_N)_{N\in\mathbb{N}}$, we have 
    $$\EIP T^n  = \sum_{t\in \sigma((n_j)_{j\in\mathbb{N}})} \omega_{\Phi}(t)\cdot P_t$$ in the strong operator topology.
    \end{corollary}

For example, for every integer $a\geq 2$, the sequence $n_j=a^{j-1}$ has rational spectrum. To check this it suffices to show that for any irrational $\alpha$, the sequence $\alpha^{n_j}$ does not converges to $1$ as $j\rightarrow\infty$.
    Considering the expansion of $\arg(\alpha)$ in base $a$, taking a power $\alpha^{n_j}$ shifts the digits $(j-1)$ times to the left.  If $\alpha^{n_j}\to 1$, this implies that for sufficiently large $j$ all entries in the decimal expansion of $\arg(\alpha)$ are $0$, contradicting that $\alpha$ is irrational. On the other hand, any sequence $(n_j)_{j\in\N}$ satisfying $\lim_{j\to\infty} n_{j+1}/n_j$ has uncountable spectrum (see~\cite[Proposition 3.5]{bergelson2014rigidity}) and so in particular is not rational. 

The next corollary is a direct application of Corollary~\ref{METrational}.
\begin{corollary} 
    Let $\XX=(X,\mathcal{B},\mu,T)$ be an invertible measure preserving system, let $a\geq 2$ be an integer, and set $n_j=a^{j-1}$. 
    \begin{enumerate}
              \item For all $f\in L^2(\mu)$ and any increasing F\o lner sequence $\Phi=(\Phi_N)_{N\in\mathbb{N}}$, the ergodic average~\eqref{eq:ergodic-average} of $f$ along $\IP_{\Phi_N}\big((n_j)_{j\in\mathbb{N}}\big)$        exists in $L^2(\mu)$. 
        \item 
        \label{item:totally}
        Furthermore, if $\XX$ is ergodic and the points $e(m/p^n)$ for $m,n\in\mathbb{N}$ and $p$ any prime divisor of $a$ do not lie in the point spectrum of $\XX$, then the ergodic average~\eqref{eq:ergodic-average} of $f$ along $\IP_{\Phi_N}\bigl((n_j)_{j\in \N}\bigr)$ converges to $\int_X f \d\mu$. 
     \end{enumerate}
\end{corollary}
Note that the hypothesis in Part~\eqref{item:totally} holds whenever the system $\XX$ is totally ergodic. 
 We use this to deduce an equidistribution result along $\IP\big((a^{j-1})_{j\in\mathbb{N}}\big)$.
\begin{corollary}[Equidistribution]
\label{cor:equidistribution}
Let $a\geq 2$ be an integer and let $\alpha\in \CS^1$ be either an irrational or be a root of unity whose order is coprime to $a$. Then for every continuous function $f\colon \CS^1\rightarrow \mathbb{C}$, we have \begin{equation}\label{equidistribution}\lim_{N\rightarrow\infty}\E_{n\in\IP_{\Phi_N\big((a^{j-1})_{j\in\mathbb{N}}\big)}} f(\alpha^n x) = \int_{\CS^1} f \d\mu
    \end{equation} for all $x\in \CS^1$.
\end{corollary}

\begin{proof}
By the Stone-Weierstrass theorem, any continuous function on a compact abelian group is a uniform limit of linear combinations of characters. Thus it suffices to prove~\eqref{equidistribution} when $f$ is a character. If $f=1$, this is immediate.  Otherwise, choose $0\not =\ell\in\mathbb{Z}$ such that $f(x)=x^\ell$.  It follows that (recall that $\omega_\Phi$ is defined in~\eqref{omega}) \begin{align*}
\lim_{N\rightarrow\infty}\sEIP f(\alpha^nx) &=\left( \lim_{N\rightarrow\infty} \sEIP (\alpha^\ell)^n \right) f(x) \\
&= \omega_{\Phi}(\alpha^\ell)\cdot f(x) =0, 
\end{align*}
and since $\int_{\CS^1} f(x) \d\mu = 0$, the statement follows.
\end{proof}
\subsection{An example of non-convergence along an $\IP$}
\label{sec:diverge}
To produce a sequence such that convergence fails, fix some irrational $\alpha\in \CS^1$ and the standard F\o lner sequence $\Phi = (\Phi_N)_{N\in\N}$ with $\Phi_N = [1,N]$ for all $N\in\N$.  We use that the orbit of $\alpha$ is dense to produce a sequence $(n_j)_{j\in\mathbb{N}}$ such that 
\begin{equation}\label{argument}\lim_{N\rightarrow\infty}\sum_{j=1}^N \arctan\left(\frac{\sin(\arg(\alpha^{n_j}))}{1+\cos(\arg(\alpha^{n_j}))}\right)
\end{equation}
diverges while
\begin{equation}\label{norm}\lim_{N\rightarrow\infty}\prod_{j=1}^N \frac{1+\cos(\arg(\alpha^{n_j}))}{2}
\end{equation}
converges to a nonzero value.  
It then follows from Theorem~\ref{criteria} the ergodic average along $\IP$ generated by $(n_j)_{j\in\N}$  diverges.

To do so, set $z_j = \frac{1+\alpha^{n_j}}{2}$ and write $z_j = r_je(i\cdot\theta_j)$ for some $r_j > 0$ and $\theta_j\in (-\frac{\pi}{2},\frac{\pi}{2})$.  For each $j\in\N$, choose $n_j$ such that $\frac{1}{j}<\theta_j<\frac{1}{j} + \frac{1}{j^2}.$ Comparing with the harmonic series we see that $\sum_{j=1}^\infty \theta_j$ diverges, and it follows that ~\eqref{argument} diverges. 
Considering~\eqref{norm}, we have that  $r_j = \sqrt{\cos(\theta_j)}$
 and taking the absolute value of the product we obtain $\lim_{N\rightarrow\infty}\prod_{j=1}^N \frac{1+\cos(\arg(\alpha^{n_j}))}{2}=\prod_{j=1}^\infty r_j$. Combining these facts, we have that 
$$\lim_{N\rightarrow\infty} \left(\prod_{j=1}^N r_j\right)^2 = \exp\big(\lim_{N\rightarrow\infty} \sum_{j=0}^N \ln(\cos(\theta_j))\big).$$
Approximating via the Taylor expansion $\ln\cos(\theta_j) = O(\theta_j^2) = O(1/j^2)$ and comparing to the convergent series $\sum_{j=0}^N \frac{1}{j^2}$, the limit $\lim_{N\rightarrow\infty} \sum_{j=0}^N \ln(\cos(\theta_j))$ exists. It follows that~\eqref{norm} is nonzero. 

\section{Characteristic factors for multiple ergodic averages along an $\IP$}
\label{sec:multiple}
\subsection{Preliminaries on characteristic factors}
\label{sec:char-factor}

Characteristic factors were implicit in the work of Furstenberg~\cite{furstenberg1977ergodic}, and were given this name by  Furstenberg and Weiss~\cite{furstenberg1996mean} in their study of some double averages. They have been used to prove the convergence of multiple ergodic averages~\cite{host2005nonconventional, ziegler2007universal}, and though there are other methods to prove the convergence (such as~\cite{walsh2012norm}), characteristic factors give more information that can be used to derive limit formulas and combinatorial corollaries. We adapt these techniques for $\IP$ averages, and start by recalling the precise definition. 
\begin{definition}[Characteristic factors]
    Let $\XX=(X,\mathcal{B},\mu,T)$, let $(A_N)_{N\in\mathbb{N}}\subseteq\mathbb{N}$ be finite subsets, and let $k\geq 1$. A factor $\YY=(Y,\CY,\nu,S)$ of $\XX$ is {\em characteristic for the average}
    \begin{equation}
        \label{eq:mult}
    \E_{n\in A_N} T^n f_1\cdot\ldots \cdot T^{kn}f_k
        \end{equation}
 if the difference between this average and the average 
    $$\E_{n\in A_N} T^n \E(f_1\mid\CY)\cdot\ldots \cdot T^{kn}\E(f_k\mid \CY)$$
    where each function is replaced by its expectation on $\YY$ 
    converges to $0$ in $L^2(\mu)$ 
    for all $f_1,\dots,f_k\in L^\infty(\mu)$. 
\end{definition}
Equivalently, the factor $\YY$ is characteristic for~\eqref{eq:mult} if this average converges to $0$ in $L^2(\mu)$ whenever there is some $i\in\{1, \dots, k\}$ such that  expectation $\E(f_i\mid \CY) = 0$. 

In~\cite{host2005nonconventional} (and later in~\cite{ziegler2007universal}), characteristic factors for the average~\eqref{eq:mult} with $(A_N)_{N\in\N}$ being a F\o lner sequence in $\N$ are described, defining a factor $Z_{k-1} = Z_{k-1}(X)$ (known as the {\em Host-Kra factor}) of the system $\XX$ for each $k\geq 1$. Our next result shows that the same factors $Z_{k-1}$ are  characteristic along  $\IP$s. 
\begin{theorem}[Characteristic factors for multiple ergodic averages along $\IP$s.]
\label{HKcharacteristic}
 Let $\XX=(X,\mathcal{B},\mu,T)$ be an invertible measure preserving system, let $(n_j)_{j\in\mathbb{N}}$ be a sequence with rational spectrum, let $\ell_1,\dots,\ell_k$ be distinct and let $k\geq 2$. For any increasing F\o lner sequence $\Phi=(\Phi_N)_{N\in\mathbb{N}}$,  the factor $Z_{k-1}(\XX)$ is characteristic for the average
 $\EIP\prod_{i=1}^k T^{\ell_in}f_i.$
\end{theorem}

When passing to a subsequence, there is no reason, apriori, that the factors $Z_k$ that are characteristic for the multiple average~\eqref{eq:mult} taken along a F\o lner in $\N$ are also characteristic for the multiple ergodic averages along  the subsequence.   In particular, for an $\IP$, this fails for even the simplest case $k=1$ that corresponds to the mean ergodic theorem, explaining the assumption that $k\geq 2$.  To prove Theorem~\ref{HKcharacteristic} in this case, we develop $\IP$ versions of the tools used in the proof in~\cite{host2005nonconventional}, including $\IP$ versions of the cubic systems, the Gowers-Host-Kra seminorms, and the associated factors.  A posteriori, we show that for $k\geq 2$, the factors agree with the factors $Z_k$.

\subsection{An $\IP$-version of the van der Corput lemma.}

We start with an $\IP$ version of a standard tool for identifying characteristic factors for multiple ergodic averages, the van der Corput Lemma.  
For a F\o lner sequence $\Phi=(\Phi_N)_{N\in\mathbb{N}}$ and fixed $M\in\mathbb{N}$, define the F\o lner sequence $\Phi^M=(\Phi_N^M)_{N\in\mathbb{N}}$ by setting $\Phi_N^M = \Phi_N\backslash \Phi_M$. It follows immediately from the definitions that $\Phi^M$ is a F\o lner sequence for every $M\in\mathbb{N}$ and that  if $\Phi$ is an increasing  F\o lner sequence then so is $\Phi^M$. 
\begin{lemma}[van der Corput lemma for increasing $\IP$-F\o lner sequences]\label{vdc}
     Let $\mathcal{H}$ be a Hilbert space with norm $\|\cdot\|$ and inner product $\langle\cdot,\cdot\rangle$, and let $(x_n)_{n\in\mathbb{N}}$ be a sequence in $\mathcal{H}$. For a sequence  $(n_j)_{j\in\mathbb{N}}$ and increasing F\o lner sequence $\Phi=(\Phi_N)_{N\in\mathbb{N}}$, we have 
    \begin{multline*}
        \limsup_{N\rightarrow \infty} \|\E_{n\in \IP_{\Phi_N}\bigl((n_j)_{j\in\mathbb{N}}\bigr)} x_{n}\|^2 \\
        \leq  \limsup_{M\rightarrow\infty} \E_{m_1,m_2\in \IP_{\Phi_M}\bigl((n_j)_{j\in\mathbb{N}}\bigr)} \limsup_{N\rightarrow\infty} \E_{n\in \IP_{\Phi^M_N}\bigl((n_j)_{j\in\mathbb{N}}\bigr)}\left<x_{n+m_1},x_{n+m_2}\right>.
        \end{multline*}
\end{lemma}
\begin{proof}
Setting  $\beta(n,m)=n+m$, we have a bijection 
\[
\beta\colon \IP_{\Phi_M}\bigl((n_j)_{j\in\mathbb{N}}\bigr)\times \IP_{\Phi_N^M}\bigl((n_j)_{j\in\mathbb{N}}\bigr)\rightarrow \IP_{\Phi_N}\bigl((n_j)_{j\in\mathbb{N}}\bigr)
\]
for all $N,M\in\N$.  It follows that for all $M\in\N$, we have 
$$\sEIP x_n = \E_{n\in \IP_{\Phi^M_N}} \E_{m\in \IP_{\Phi_M}} x_{n+m}.$$
    By Jensen's inequality, we have that 
    \begin{align*}\left\|\E_{n\in \IP_{\Phi^M_N}} \E_{m\in \IP_{\Phi_M}} x_{n+m}\right\|^2 &\leq \E_{n\in \IP_{\Phi^M_N}}\left\|\E_{m\in \IP_{\Phi_M}}x_{n+m}\right\|^2  \\&=\E_{n\in \IP_{\Phi_N^M}}\E_{m_1,m_2\in \IP_{\Phi_M}}\left<x_{n+m_1,n+m_2}\right>.
    \end{align*} 
    Interchanging the order of summation and taking $N\rightarrow\infty$, we obtain that for all $M\in\mathbb{N}$, 
\[
\limsup_{N\rightarrow\infty}\|\sEIP x_n\|^2 \leq \E_{m_1,m_2\in \IP_{\Phi_M}} \limsup_{N\rightarrow\infty} \E_{n\in\IP_{\Phi_N^M}}\left<x_{n+m_1},x_{n+m_2}\right>.
\]
    The statement follows as $M\to \infty$. 
\end{proof}
 
 Combining this lemma and induction, we have the generalization of Corollary~\ref{cor:equidistribution} to polynomial sequences.
    \begin{corollary}[Equidistribution along polynomials]
Let $a\geq 2$ be an integer and let $\alpha\in \CS^1$ be either an irrational or be a root of unity whose order is coprime to $a$. Then for every $k\geq 1$, integer valued (non-constant) polynomial $p\colon \N\rightarrow \N$ of degree $k$, and continuous function $f\colon \CS^1\rightarrow \mathbb{C}$, we have \begin{align*}\lim_{N\rightarrow\infty}\E_{n\in\IP_{\Phi_N((a^{j-1})_{j\in\mathbb{N}})}} f(\alpha^{p(n)} x) = \int_{\CS^1} f \d\mu
    \end{align*} for all $x\in \CS^1$.
\end{corollary}

To control  multiple ergodic averages along an $\IP$ associated to a  sequence with rational spectrum, we introduce the relevant  version of the uniformity seminorms, defined inductively analogous to the construction in~\cite{host2005nonconventional}. 

\begin{definition}[The (rational) $\IP$ Gowers-Host-Kra seminorms]
\label{def:seminorms}
    Let $\XX=(X,\mathcal{B},\mu,T)$ be an invertible measure preserving system and  $f\in L^\infty(\mu)$.  For $k=1$, define the {\em first seminorm $\|\cdot\|_{\U^1_{\IP}(\XX,T)}$} by setting 
    $$\|f\|_{\U^1_{\IP}(\XX,T)} := \sup_{(n_j)_{j\in\mathbb{N}}}\sup_{\Phi}\lim_{N\rightarrow\infty}\left\| \EIP T^n f\right\|_{L^2(\mu),}$$ where the supremums are taken over all sequences $(n_j)_{j\in\mathbb{N}}$ with rational spectrum and all increasing $\IP$-F\o lner sequences $\Phi$. For $k\geq 1$, define the {\em $(k+1)$\textsuperscript{st}  seminorm $\|\cdot\|_{\U^{k+1}_{\IP}(\XX,T)}$} by setting  
  $$\|f\|_{\U^{k+1}_{\IP}(\XX,T)} := \sup_{(n_j)_{j\in\mathbb{N}}}\sup_{\Phi}\lim_{N\rightarrow\infty} \left(\E_{m_1,m_2\in \IP_{\Phi_N}\bigl((n_j)_{j\in\mathbb{N}}\bigr)} \left\|T^{m_1}f\cdot T^{m_2}\overline{f}\right\|_{\U^k_{\IP}(\XX,T)}^{2^{k}}\right)^{1/2^{k+1}},$$
  again taking the supremums  over all sequences $(n_j)_{j\in\mathbb{N}}$ with rational spectrum and all increasing $\IP$-F\o lner sequences $\Phi$.
\end{definition}

As in the case of a F\o lner sequence in $\N$, these seminorms are characterized in terms of certain measures,  inductively constructed  analogously to how the cubic measures are constructed in~\cite{host2005nonconventional}.  Using notation of~\cite{host2005nonconventional}, set $V_k=\{0,1\}^k$ to be the $k$-dimensional cube. We write a vertex $\epsilon\in V_k$ without commas and set  $|\epsilon|=\sum_{i=1}^k \epsilon_i$.  If  $\epsilon\in V_k$ and $\eta\in V_\ell$, we  concatenate them to obtain an element $\epsilon\eta\in V_{k+\ell}$.  Let $\mathcal{C}\colon \mathbb{C}\rightarrow\mathbb{C}$ denote  complex conjugation.  
In~\cite{host2005nonconventional}, joinings are built inductively over the invariant factor, and in our inductive construction the invariant factor is replaced by the rational Kronecker factor.  

\begin{definition}[The (rational) $\IP$ cubic measures]
\label{def:cubic-measure}
    Let $\XX=(X,\mathcal{B},\mu,T)$ be an invertible measure preserving system, set  $X^{[0]}=X$, and set $\tilde{\mu}^{[0]}=\mu$. For $k\geq 1$,  identifying $X^{[k+1]}$ with $X^{[k]}\times X^{[k]}$, we define the {\em cubic measure }  $\tilde{\mu}^{[k+1]}$ on $\XX^{[k+1]}$ to be the relatively independent joining of $\XX^{[k]}$ with itself over the factor $\mathcal{K}_{\rat}(\XX^{[k]})$.  
\end{definition}
This means that if $(f_\epsilon)_{\epsilon\in V_{k+1}}$ are $2^{k+1}$ bounded functions on $X$, the measure $\tilde{\mu}^{[k+1]}$ is defined by 
\[
\int_{X^{[k+1]}}\bigotimes_{\epsilon \in V^{k+1}}f_\epsilon\d\tilde{\mu}^{[k+1]} = 
\int_{X^{[k]}}\E\Big(\bigotimes_{\eta\in V_k} f_{\eta 0}\mid \mathcal{K}_{\rat}
\Big)\cdot \E\Big(\bigotimes_{\eta\in V_k} f_{\eta 1}\mid \mathcal{K}_{\rat}
\Big)\d\tilde{\mu}^{[k]}
.
\]

We check that rational $\IP$-cubic measures control  $\IP$-semirnorms.
\begin{proposition}
\label{measurecontrol}
    Let $\XX=(X,\mathcal{B},\mu,T)$ be an invertible measure preserving system and let $k\geq 1$. For every $f\in L^\infty(\mu)$, we have

    $$\|f\|_{\U^{k}_{\IP(\XX,T)}}^{2^{k}} \leq \int_{X^{[k]}} \bigotimes_{\epsilon\in V_k} \mathcal{C}^{|\epsilon|} f \d\tilde{\mu}^{[k]}.$$
\end{proposition}
\begin{proof}
    We proceed by induction on $k$. For $k=1$, Corollary~\ref{METrational} implies that
    \begin{equation}
    \label{eq:k=1bound}
    \|f\|^2_{\U^1_{\IP}(\XX,T)} = \sup_{(n_j)_{j\in\mathbb{N}}}\sup_{\Phi}\Big\|\sum_{t\in \sigma((n_j)_{j\in\mathbb{N}})} \omega_{\Phi}(t)\cdot P_t(f)\Big\|_{L^2(\mu)} ^2,
    \end{equation}
    where 
$\omega_{\Phi}(t) = \lim_{N\to\infty}\prod_{j\in\Phi_N}\frac{1+t^{n_j}}{2}.$
Since eigenspaces corresponding to different eigenvalues are orthogonal and  $|\omega_{\Phi}(t)|\leq 1$ for all $t$ independently both of the sequence $(n_j)_{j\in\mathbb{N}}$ and of the increasing F\o lner sequence $\Phi=(\Phi_N)_{N\in\mathbb{N}}$, the quantity in~\eqref{eq:k=1bound} is bounded by 
  \[
  \Big\|\sum_{\{t=e^{2\pi i r}: r\in\Q\}}P_t (f)\Big\|_{L^2(\mu)}^2 = \|\E(f\mid \mathcal{K}_{\rat}(\XX))\|_{L^2(\mu)}^2.  
  \]
    On the other hand,
    $$\int_{X^{[1]}} f\otimes \overline{f} \d\tilde{\mu}^{[1]} = \int_{X} \E(f\mid \mathcal{K}_{\rat}(\XX))\cdot \overline{\E(f\mid\mathcal{K}_{\rat}(\XX))}\d\mu = \|\E(f\mid \mathcal{K}_{\rat}(\XX))\|^2_{L^2(\mu),}$$ completing the base case.

  For $k\geq 2$, by the inductive  hypothesis, Corollary~\ref{METrational} and the same argument used in the base case, we have
     \begin{align*}\|f\|_{\U^{k}_{\IP}(\XX,T)}^{2^k} &= \sup_{(n_j)_{j\in\mathbb{N}}}\sup_{\Phi}\lim_{N\rightarrow\infty} \left(\E_{m_1,m_2\in \IP_{\Phi_N}\bigl((n_j)_{j\in\mathbb{N}}\bigr)} \left\|T^{m_1}f\cdot \overline{T^{m_2} f}\right\|_{\U^{k-1}_{\IP}(\XX,T)}^{2^{k-1}}\right) \\&\leq \sup_{(n_j)_{j\in\mathbb{N}}}\sup_{\Phi}\lim_{N\rightarrow\infty} \E_{m_1,m_2\in\IP_{\Phi_N}\bigl((n_j)_{j\in\mathbb{N}}\bigr)} \int_{X^{[k-1]}}\bigotimes_{\epsilon\in V_{k-1}} \mathcal{C}^{|\epsilon|}T^{m_1}f\cdot \overline{T^{m_2} f} \d\tilde{\mu}^{[k-1]} \\&= \sup_{(n_j)_{j\in\mathbb{N}}}\sup_{\Phi}\lim_{N\rightarrow\infty}  \int_{X^{[k-1]}}\Big|\E_{n\in\IP_{\Phi_N}\bigl((n_j)_{j\in\mathbb{N}}\bigr)}\bigotimes_{\epsilon\in V_{k-1}} \mathcal{C}^{|\epsilon|}T^{n}f \Big|^2 \d\tilde{\mu}^{[k-1]}
     \\&=\sup_{(n_j)_{j\in\mathbb{N}}}\sup_\Phi \lim_{N\rightarrow\infty} \int_{X^{[k-1]}} \Big|\E_{n\in\IP_{\Phi_N}\bigl((n_j)_{j\in\mathbb{N}}\bigr)}(T^{[k-1]})^n(\bigotimes_{\epsilon\in V_{k-1}} \mathcal{C}^{|\epsilon|}f)\Big|^2\d\tilde{\mu}^{[k-1]}\\&\leq \int_{X^{[k-1]}} \E\Big(\bigotimes_{\epsilon\in V_{k-1}} \mathcal{C}^{|\epsilon|}f\mid\mathcal{K}_{\rat}(X^{[k-1]}\Big)\cdot \overline{\E\Big(\bigotimes_{\epsilon\in V_{k-1}} \mathcal{C}^{|\epsilon|}f \mid \mathcal{K}_\rat(X^{[k-1]}\Big)}\d\tilde{\mu}^{[k-1]} \\&= \int_{X^{[k]}} \bigotimes_{\epsilon\in V_k} \mathcal{C}^{|\epsilon|}f \d\tilde{\mu}^{[k]},
     \end{align*}
    completing the argument.
\end{proof}

We state a simple lemma on rational spectrum. 
\begin{lemma}\label{powerrational}
    If the spectrum of a sequence $(n_j)_{j\in \mathbb{N}}$ is rational, then for every integer $a\in\mathbb{N}$, the spectrum of the sequence $(a\cdot n_j)_{j\in\mathbb{N}}$ is also rational.
\end{lemma}
\begin{proof}
   If $\alpha\in \sigma(a\cdot (n_j)_{j\in\mathbb{N}})$, then $\alpha^a \in \sigma((n_j)_{j\in\mathbb{N}})$ and so $\alpha$ is rational.
\end{proof}
We check that $\IP$-seminorms control multiple ergodic averages along an $\IP$, and note that it then  follows immediately from Proposition~\ref{measurecontrol} that rational $\IP$ cubic measures control multiple ergodic averages along an $\IP$.
\begin{proposition}
\label{seminormcontrol}
    Let $(n_j)_{j\in\mathbb{N}}$ be a sequence with rational spectrum. Let $\XX=(X,\mathcal{B},\mu,T)$ be an invertible measure preserving system, let $k\geq 1$,  let $f_1,\dots,f_k\in L^\infty (\mu)$ with $\|f_i\|_{L^\infty(\mu)} \leq 1$  for $i=1, \dots, k$, and let $\ell_1,\dots,\ell_k$ be distinct integers.  For any increasing F\o lner sequence $\Phi=(\Phi_N)_{N\in\mathbb{N}}$, we have
\begin{equation}\label{seminormcontroleq}\limsup_{N\rightarrow\infty}\Big\|\E_{n\in \IP_{\Phi_N}\bigl((n_j)_{j\in\mathbb{N}}\bigr)} T^{\ell_1n}f_1\cdot\ldots\cdot T^{\ell_kn}f_k \Big\|_{L^2(\mu)} \leq \min_{1\leq i \leq k} \|f_i \|_{\U^{k}_{\IP}(\XX,T)}.
    \end{equation}
\end{proposition}

\begin{proof}
    We proceed by induction on $k$, noting that the case $k=1$ follows from Lemma~\ref{powerrational} and Theorem~\ref{METrational}. 
    Fixing $k\geq 2$ and changing the order of the functions $f_1,\dots,f_k$ if needed, it suffices to show that the left hand side of~\eqref{seminormcontroleq} is bounded by $\|f_1\|_{\U^{k+1}_\IP(\XX,T)}.$
    Setting
    $$x_n = T^{\ell_1n}f_1\cdot\ldots\cdot T^{\ell_kn}f_k$$ and applying Lemma~\ref{vdc}, we have that 
\begin{align*}&\limsup_{N\rightarrow\infty}\Big\|\underset{n\in \IPPhi}{\E} x_n\Big\|_{L^2(\mu)}^2  \\ &  \leq  \limsup_{M\rightarrow\infty}
\underset{m_1,m_2\in \IP_{\Phi_M}}{\E} \limsup_{N\rightarrow\infty} \underset{n\in\IP_{\Phi_N^M}}{\E} \int_X \prod_{i=1}^k T^{\ell_i(n+m_1)} f_i \cdot \prod_{i=1}^k  \overline{ T^{\ell_i(n+m_2)}f_i} \d\mu \\
& = \limsup_{M\rightarrow\infty}
\underset{m_1,m_2\in \IP_{\Phi_M}}{\E}\limsup_{N\rightarrow\infty} \underset{n\in\IP_{\Phi_N^M}} {\E}  
\int_X \left(T^{\ell_km_1}f_k  \overline{ T^{\ell_km_2} f_k}\right) \prod_{i=1}^{k-1} T^{(\ell_i-\ell_k)n} \left(T^{\ell_i m_1} f_i\cdot  \overline{T^{\ell_i m_2}f}_i\right) \d\mu, \end{align*}
where the last equality follows from 
$T$-invariance of $\mu$.  
Applying the Cauchy-Schwartz Inequality 
and using that $\|f_k\|_{L^\infty(\mu)}\leq 1$, this last quantity is bounded by 
\begin{align*}\limsup_{M\rightarrow\infty}\E_{m_1,m_2\in \IP_{\Phi_M}} \limsup_{N\rightarrow\infty} \Big\|\E_{n\in\IP_{\Phi_N^M}} \prod_{i=1}^{k-1} T^{(\ell_i-\ell_k)n} \left(T^{\ell_i m_1} f_i\cdot T^{\ell_i m_2}\right)\Big\|_{L^2(\mu)}.  \end{align*}
  By the induction hypothesis, we bound this by    $$\limsup_{M\rightarrow\infty}\E_{m_1,m_2\in \IP_{\Phi_M}} \|T^{\ell_1m_1}f_1\cdot T^{\ell_1m_2} \overline{f_1}\|_{\U^{k-1}_{\IP}(\XX,T)}$$ 
  and by convexity of the map $x\mapsto x^{2^{k-1}}$,  this is bounded by 
\begin{align*}\limsup_{M\rightarrow\infty}\big(\E_{m_1,m_2\in \IP_{\Phi_M}} \|T^{\ell_1m_1}f_1\cdot T^{\ell_1m_2} \overline{f_1}\|^{2^{k-1}}_{\U^{k-1}_{\IP}(\XX,T)}\big)^{1/2^{k-1}} \leq \|f_1\|^2_{\U^{k}_{\IP}(\XX,T)}. 
   \end{align*}
Thus we shown that 
   \[
   \limsup_{N\rightarrow\infty}\|\E_{n\in \IP_{\Phi_N}\bigl((n_j)_{j\in\mathbb{N}}\bigr)} T^{\ell_1n}f_1\cdot\ldots\cdot T^{\ell_kn}f_k \|_{L^2(\mu)}^2 \leq \min_{1\leq i \leq k} \|f_1 \|^2_{\U^{k}_{\IP}(\XX,T)},
   \] and taking square roots, the statement follows. 
   \end{proof}

\begin{lemma}\label{abscont}
Let $\XX=(X,\mathcal{B},\mu,T)$ be an invertible measure preserving system, let $k\geq 2$, and let    $(f_\epsilon)_{\epsilon\in V_k}\in L^\infty(\mu)$ be  $2^k$ functions.  If  
$\|f_\epsilon\|_{\U^{k}(\XX,T)}=0$ for some 
$\epsilon\in V_k$, then  
$$\int_{X^{[k]}} \bigotimes_{\epsilon\in V_k} \mathcal{C}^{|\epsilon|}f_\epsilon \d\tilde{\mu}^{[k]}=0.$$ 
\end{lemma}

\begin{proof}
Let $\YY=(Y,\mathcal{C},\nu,S)$ be an invertible measure preserving system. 
For each eigenfuction $f\in L^2(\YY)$  with rational eigenvalue, there exists some $\ell$ such that $f$ is $S^\ell$-invariant. Letting $\mathcal{I}_\ell$ denote the $\sigma$-algebra of $\ell$-invariant functions, the sequence of $\sigma$-algebras $\mathcal{I}_{\ell!}$ is increasing and $\mathcal{K}_\rat(\YY)$ is a sub-$\sigma$-algebra of the join $\bigvee_{\ell\in\mathbb{N}} \mathcal{I}_{\ell!}$.

Conversely, for each $\ell\in \mathbb{N}$, taking $f$ to be an $S^\ell$-invariant function and letting $\xi_\ell$ be the first root of unity of order $\ell$, we have that the functions 
$$f_j= \sum_{i=0}^{\ell-1} \xi_\ell^{i\cdot j} \cdot T^if$$ are eigenfunctions for all $1\leq j \leq \ell$.
Since $\sum_{i=0}^{\ell-1} \xi_\ell^{i\cdot j} = 0$ for all $1\leq j \leq \ell-1$, we deduce that $f=\frac{1}{\ell}\sum_{j=1}^\ell f_j$. Thus the $\sigma$-algebra $\mathcal{I}_\ell$ of $\ell$-invariant functions is a sub-$\sigma$-algebra of $\mathcal{K}_{\rat}(\YY)$, and 
$\mathcal{K}_{\rat}(\YY) = \bigvee_{\ell\in\mathbb{N}}\mathcal{I}_{\ell!}$. In particular, for any function $f\in L^\infty(\YY)$, we have that 
\[\E(f\mid\mathcal{K}_\rat(\YY)) = \lim_{\ell\rightarrow\infty} \E(f\mid\mathcal{I}_{\ell!}).
\]

Taking the Host-Kra measure $\mu_\ell^{[k]}$ on $\XX=(X,\mathcal{B},\mu,T^{\ell!})$ (see~\cite[Section 3]{host2005nonconventional}), then it follows from this and the definition of $\tilde{\mu}^{[k]}$ that for any function $f\in L^\infty(\mu)$ we have
$$\int_{X^{[k]}} \bigotimes_{\epsilon\in V_k} \mathcal{C}^{|\epsilon|}f  \d\tilde{\mu}^{[k]} = \lim_{\ell\rightarrow\infty} \int_{X^{[k]}} \bigotimes_{\epsilon\in V_k} \mathcal{C}^{|\epsilon|}f~ \d\mu_{\ell}^{[k]}.$$
By a theorem of Leibman~\cite[Theorem 2]{leibman-hkz}, it follows that  $\|f\|_{\U^k(\XX,T)}=0$ is equivalent to 
\[\int_{X^{[k]}} \bigotimes_{\epsilon\in V_k} \mathcal{C}^{|\epsilon|}f\d\mu_{\ell}^{[k]}=0 \]
for all $\ell\in\N$.   It follows that 
\begin{equation*}
\int_{X^{[k]}} \bigotimes_{\epsilon\in V_k} \mathcal{C}^{|\epsilon|}f\d\tilde{\mu}^{[k]}=0 \qedhere
\end{equation*}
\end{proof}

We have now assembled the tools to complete the proof of Theorem~\ref{HKcharacteristic}.
\begin{proof}[Proof of Theorem~\ref{HKcharacteristic}]
Let $f_1,\dots,f_k\in L^\infty(\mu)$ and assume  $\E(f_j \mid Z_{k-1}(\XX))=0$ for some $1\leq j \leq k$.  It follows from this assumption that $\|f_j\|_{\U^{k}(\XX,T)}=0$ (see for example~\cite[Lemma 4.3]{host2005nonconventional}).  
By Propositions~\ref{measurecontrol} and~\ref{seminormcontrol},  we have that 
$$\lim_{N\rightarrow\infty}\Big\|\sEIP \prod_{i=1}^k T^{\ell_in} f_i\Big\|^{2^{k}}_{L^2(\mu)} \leq \int_{X^{[k]}} \bigotimes_{\epsilon\in V_k} f_j \d\tilde{\mu}^{[k]}.$$
By Lemma~\ref{abscont}, we deduce that $$\lim_{N\rightarrow\infty}\big\|\sEIP \prod_{i=1}^k T^{\ell_in} f_i\big\|^{2^{k}}_{L^2(\mu)} =0$$ and therefore
\begin{equation*}
    \lim_{N\rightarrow\infty}\sEIP \prod_{i=1}^k T^{\ell_in} f_i=0.\qedhere\end{equation*}
\end{proof}

\section{Limit formula in nilsystems}\label{limit:sec}
\subsection{Limits for systems in which the spectrum has no nontrivial eigenvalues}
Nilsystems are the natural class that arise in understanding uniformity norms and cubes in~\cite{host2005nonconventional} and generalizations have been to understand combinatorial structures associated to arithmetic progressions and cubes, including work in~\cite{cgss-2023,candela-szegedy-inverse, candela-szegedy-memoir, gmv,gmv2,gmv3,jst,jt21-1}. 

 Recall that $\XX = (X, \CB, \mu, T_\tau)$ is a {\em $k$-step nilsystem} if $X = G/\Gamma$ for a $k$-step nilpotent Lie group $G$ with discrete, cocompact  subgroup $\Gamma$, $\CB$ is the Borel $\sigma$-algebra, $\mu$ is the unique Borel probability measure on $X$ invariant under the action of $G$ on $X$ by left translation, and $T_
\tau\colon X\to X$ is the action $x\mapsto \tau\cdot x$ for some fixed element $\tau\in G$. In case of ambiguity, we include the space on the measure and write $\mu = \mu_X = \mu_{G/\Gamma}$. 

As with the ergodic average, the limit formula in a nilsystem depends crucially on the spectrum of $T_\tau$. 
We start with the special case of a $1$-step nilsystem whose spectrum is disjoint from $\sigma\big((n_j)_{j\in\mathbb{N}}\big)$.
Recall that $\ZZ=(Z, \mathcal{B},\mu, R_a)$ is a {\em $1$-step nilsystem} means that $Z$ is a compact abelian group, $\CB$ is the Borel $\sigma$-algebra, $\mu$ is the Haar measure, and $R_a$ is the rotation by a fixed element $a\in Z$.  This system is ergodic if and only if the orbit of the element $a$ is dense. 

\begin{theorem}[Limit formula for two terms]\label{Kroneckerformula}
    Let $\ZZ=(Z, \mathcal{B},\mu, R_a)$ be an ergodic $1$-step nilsystem, let $(n_j)_{j\in\mathbb{N}}$ be a sequence whose spectrum contains no nontrivial eigenvalue of $Z$, and let $\Phi=(\Phi_N)_{N\in\mathbb{N}}$ be an increasing F\o lner sequence. 
    Then for all  $f_1,\dots,f_k\in L^\infty(\mu)$ and integers $\ell_1,\dots,\ell_k\in \mathbb{N}$, we have
    $$\lim_{N\rightarrow\infty} \E_{n\in \IP_{\Phi_N}\big((n_j)_{j\in\mathbb{N}}\big)} \prod_{i=1}^k f_i(T^{\ell_in}x) = \int_Z \prod_{i=1}^k f_i(x+\ell_it) \d\mu(t)$$ in $L^2(\mu)$.
\end{theorem}
\begin{proof}
Since linear combinations of characters are dense in $L^2(\mu)$, it it suffices to establish the formula when  $f_1,\dots,f_k$ are characters of $Z$.  Furthermore, since $Z$ is an abelian group, it suffices to assume that they are  eigenfunctions with eigenvalues $\alpha_1,\dots\alpha_k$ respectively. Recall that $\omega_{\Phi}$ is defined in~\eqref{omega} and $\delta_1$ in~\eqref{delta}. Since the spectrum of $(n_j)_{j\in\mathbb{N}}$ contains no nontrivial eigenvalue of $Z$, we have that $\omega_{\Phi}(\alpha) = \delta_1(\alpha)$ for all eigenvalues of the system. 
By Theorem~\ref{METrational}, and Fourier analysis it follows that 
\begin{align*}
  \lim_{N\rightarrow\infty} \sEIP & \prod_{i=1}^k f_i(T^{\ell_i n}x) =  \lim_{N\rightarrow\infty} \sEIP \Big(\prod_{i=1}^k \alpha_i^{\ell_i}\Big)^n \cdot\Big(\prod_{i=1}^k f_i(x)\Big) \\&= \omega_{\Phi}\Big(\prod_{i=1}^k \alpha_i^{\ell_i}\Big)\cdot \Big(\prod_{i=1}^k f_i(x)\Big) = \delta_1\Big({\prod_{i=1}^k \alpha_i^{\ell_i}}\Big)\cdot \Big(\prod_{i=1}^k f_i(x)\Big)\\&=\int_Z \prod_{i=1}^k f_i(x+\ell_it)\d\mu(t),
\end{align*} 
completing the proof.
\end{proof}

For a nilsystem $\XX=(G/\Gamma, \mathcal{B},\mu, R_a)$ and integers $\ell_1,\dots,\ell_k$, set
$$\tilde{G} = \tilde{G}_{\ell_1,\dots,\ell_k} = \Big\{ \Big(\prod_{i=1}^{k-1} g_i^{\binom{\ell_j}{i}}u_j\Big)_{j\in\{1,\dots,k\}} :  g_i\in G_i, u_j\in G_{k}, 1\leq i,j\leq k\Big\} \leq G^k,$$ where $\binom{n}{m}=0$ if $m>n$.  (As $\ell_1,\dots,\ell_k$ are fixed, we simplify the notation and write $\tilde{G}=\tilde{G}_{\ell_1,\dots,\ell_k}$.)   Setting $\tilde{\Gamma} = \tilde{G} \cap \Gamma^{k}$, it is shown in~\cite{leibmanpolymap} that $\tilde{G}$ is a closed subgroup of $G^k$. The discrete subgroup $\tilde{\Gamma}$ is cocompact and so $\tilde{X} = \tilde{G}/\tilde{\Gamma}$ is a nilmanifold endowed with a Haar measure $\tilde{\mu}$.

The main goal of this section is prove the following limit formula. 
\begin{theorem}\label{formula}
    Let $k\geq 1$, let $\XX=(G/\Gamma, \mathcal{B},\mu_{G/\Gamma}, T_\tau)$ be an ergodic $k$-step nilsystem, let $(n_j)_{j\in\mathbb{N}}$ be a sequence whose spectrum contains no nontrivial eigenvalue of $T_\tau$, and let $(\Phi_N)_{N\in\mathbb{N}}$ be an increasing F\o lner sequence. For all $f_1,\dots,f_k\in L^\infty(\mu)$ and distinct integers $\ell_1,\dots,\ell_k\in\mathbb{N}$, we have 
$$\lim_{N\rightarrow\infty} \E_{n\in \IP_{\Phi_N}\bigl((n_j)_{j\in\mathbb{N}}\bigr)} \prod_{i=1}^k f_i(T_\tau^{\ell_in}x) = \int_{\tilde{X}} \prod_{i=1}^k f_i(xy_i) \d\tilde{\mu}(y\tilde{\Gamma})$$ in $L^2(\mu)$,  where $y=(y_1,\dots,y_k)\in \tilde{G}$.
\end{theorem}

For general (non-$\IP$) multiple averages Ziegler~\cite{zieglerformula} derived a limit formula, and a different method for doing so was given  in~\cite{BHK}. Both proofs rely on the unique ergodicity of nilsystems and the existence of the limit of the mean ergodic average everywhere for continuous functions on nilsystems.  
 In our context of evaluating the averages along $\IP$-F\o lner sequences, we need a different approach.  Instead, we identify a class of  functions that is dense and  for which we can show convergence everywhere, and this class suffices to prove for convergence in $L^2(\mu)$ for all bounded functions.

We start with the case where the set of eigenvalues of the nilsystem is disjoint from the spectrum of $(n_j)_{j\in\N}$. 
\begin{theorem}[Pointwise convergence along $\IP$s for nilsystems with disjoint spectrum]\label{IPpointwise}
    Let $k\geq 1$ and let $\XX=(G/\Gamma,\mathcal{B},\mu,T_\tau)$ be an ergodic $k$-step nilsystem. There exists an ergodic system $\YY=(Y,\mathcal{C},\nu,T)$ on a compact metric space $Y$ such that the set of $T$-periodic points has $\nu$-measure zero and there exists a measure theoretic isomorphism $\pi\colon X\rightarrow Y$
    such that for any sequence $(n_j)_{j\in\mathbb{N}}$ whose spectrum contain no nontrivial eigenvalue of $T_\tau$ 
    and every increasing F\o lner sequence $\Phi=(\Phi_N)_{N\in\N}$, we have that 
    \begin{equation}
        \label{eq:conv-cont}
\lim_{N\rightarrow\infty}\E_{n\in \IP_{\Phi_N}\big((n_j)_{j\in\mathbb{N}}\big)} f(T_\tau^ny) = \int_Y f\d\mu
    \end{equation}
    for any continuous function $f\in C(Y)$  and  $y\in Y$. 
\end{theorem}
In general, the isomorphism $\pi\colon X\to Y$ need not be continuous. The hypotheses, and in particular the assumption that the $T$-periodic points have $\nu$-measure zero, is used to apply a theorem of Lindenstrauss~\cite[Theorem 3.1]{lindenstraussdistal}. 

\begin{proof}
We proceed by induction on the step $k$. When $k=1$, $\XX$ is an ergodic rotation on a compact abelian group $Z=G/\Gamma$ and we can take $\YY = \XX$ and $\pi\colon X\rightarrow Y$ to be the identity map. In this case, there are at most finitely many $T$-periodic points in $Y$. 
Let $p\colon Z\rightarrow \mathbb{C}$ be a linear combination of characters in $Z$ and write  
$$p=\sum_{i=1}^n b_i \xi_i$$ for some $b_i\in\mathbb{C}$ and $\xi_i\in\widehat{Z}$. Letting $a\in Z$ denote the rotation induced by $T_\tau$, we then have that for all $y\in Y$,
$$\lim_{N\rightarrow\infty} \E_{n\in \IP_{\Phi_N}\big((n_j)_{j\in\mathbb{N}}\big)} p(T^ny) = \sum_{i=0}^n b_i\left(\lim_{N\rightarrow\infty} \E_{n\in \IP_{\Phi_N}\bigl((n_j)_{j\in\mathbb{N}}\bigr)} \xi_i(a)^n\right)\xi_i(y).$$  
By assumption, no nontrivial eigenvalue of $Z$ lies in the spectrum of $(n_j)_{j\in\mathbb{N}}$, and so for each $i$, the limit on the right hand side is zero unless $\xi_i=0$ and in this case the limit is $1$. In both cases, the limit exists and is equal to $b_0=\int_Y p(y) \d\mu$. 
Taking $f\in C(Y)$ and $\varepsilon>0$, there exists a finite linear combination of characters $p\colon Z\rightarrow \mathbb{C}$ such that $\|f-p\|_\infty <\varepsilon/2$. 
Fixing some $y\in Y$, using this approximation of $f$ by $p$, and noting that by choice we have $X = Y$ and $T$ is the rotation induced by $a\in Z$, we have 
\begin{align*}
\limsup_{N\rightarrow\infty} & \Big|\E_{n\in\IP_{\Phi_N}} f(T^n y) - \int_Y f \d\mu \Big|\\ \leq & \limsup_{N\rightarrow\infty} \left|\E_{n\in\IP_{\Phi_N}} f(a^n y) - \E_{n\in\IP_{\Phi_N}} p(a^n y)\right| \\
& + \limsup_{N\rightarrow\infty}\Big|\E_{n\in\IP_{\Phi_N}} p(a^n y) - \int_Z p \d\mu \Big| + \Big|\int f \d\mu - \int p \d\mu\Big| \\ 
 \leq  &  2 \|f-p\|_\infty < \varepsilon.
\end{align*}
Since $\varepsilon>0$ is arbitrary, we deduce that \[
\lim_{N\rightarrow\infty} \E_{n\in \IP_{\Phi_N}\big((n_j)_{j\in\mathbb{N}}\big)} f(T^n y) = \int_Y f\d\mu, 
\] 
completing the case $k=1$. 

Fix some $k\geq 2$. 
Combining~\cite[Theorem 10.3 and Section 11]{host2005nonconventional}, since $\XX$ is an ergodic $k$-step nilsystem, it is measure theoretically isomorphic to a $k$-step toral system, meaning that we can assume that the Kronecker factor $Z_1(\XX)$ is a compact abelian Lie group and each nilfactor $Z_{j+1}(\XX)$ is an extension of $Z_j(\XX)$ by a torus for $1 \leq j < k$.   Thus we can write $X=Z_{k-1}(\XX)\times_\rho U_k$, for some torus $U_k$ and cocycle $\rho\colon Z_{k-1}(\XX)\to U_k$.
Applying the induction hypothesis to $Z_{k-1}(\XX)$, there exists a system $\YY_{k-1}=(Y_{k-1},\mathcal{C}_{k-1},\nu_{k-1},T_{k-1})$ that is isomorphic to $Z_{k-1}(\XX)$ satisfying the properties of the theorem, meaning that $\XX$ is measure theoretically isomorphic to an abelian extension of $\YY_{k-1}$ by a torus $U_k$ via a cocycle $\rho\colon Y_{k-1}\rightarrow U_k$. Using a theorem of Lindenstrauss~\cite[Theorem 3.1]{lindenstraussdistal}, $\rho$ is cohomologous to a continuous cocycle. As cohomologous cocycles give rise to isomorphic extensions, without loss of generality we may assume that $\rho$ is continuous and write $Y=Y_{k-1}\times_\rho U_k$ and let $T\colon Y\rightarrow Y$ denote the action induced by the cocycle $\rho$. Letting $P_{k-1}$ denote the set of periodic points in the system $\YY_{k-1}$, then the set of periodic points in  the system $\YY$ is a subset of $P_{k-1}\times U_k$ and is therefore of measure $0$. We are left with showing that $\YY$ satisfies the remaining properties in the statement. 

By Lemma~\ref{dense}, the space of continuous vertical characters on $\YY$ is dense with respect to the uniform norm.  Thus it suffices to prove the limit formula for functions of the form $f_\chi(y,u) = g(y)\chi(u)$ for some character $\chi\colon U_k\rightarrow \CS^1$ and  continuous function $g\colon Y_{k-1}\rightarrow \C$.
If $\chi$ is trivial, then the convergence follows from the induction hypothesis.  Otherwise, we have
\begin{align*}
\limsup_{N\rightarrow\infty}  \left| \sEIP f_\chi(T^n (y,t))\right| = \limsup_{N\rightarrow\infty}\left| \sEIP g(T_{k-1}^n y)\chi\circ\rho(n,y)f_\chi(y,t)\right|.
\end{align*}
By the van der Corput inequality (Lemma~\ref{vdc}),  we can bound the square of the right hand side by 
\begin{multline}
\label{eq:expanded}
  \limsup_{M\rightarrow\infty} \E_{m_1,m_2\in \IP_{\Phi_M}} \limsup_{N\rightarrow\infty} \E_{n\in\IP_{\Phi_N^M}} g(T_{k-1}^{n+m_1} y)\cdot  \\ \overline{g(T_{k-1}^{n+m_2}y)}\cdot  \rho(m_1,T^n y)\cdot \overline{\rho(m_2,T^n y)}, 
\end{multline}
where $\Phi_N^M= \Phi_N\backslash \Phi_M$ and we have used the cocycle identity $\rho(n+m,y) = \rho(n,y)\rho(m,T^n y)$ for all $m,n\in\mathbb{N}$ and $y\in Y$. 
Since $\rho$ and $g$ are continuous, using the inductive assumption applied to the convergence in~\eqref{eq:conv-cont}, we can rewrite the quantity in~\eqref{eq:expanded} as 
\begin{align*}
    \limsup_{M\rightarrow\infty} \E_{m_1,m_2\in \IP_{\Phi_M}} \int_{Y_{k-1}} g(T_{k-1}^{m_1}(y))&\overline{g(T_{k-1}^{m_2}y)} \rho(m_1,y)\cdot \overline{\rho(m_2,y)}\d\mu(y)\\  &=  \limsup_{M\rightarrow\infty}\left\|\E_{m\in \IP_{\Phi_M}} g(T_{k-1}^m \cdot ) \rho(m,\cdot)\right\|^2_{L^2(\YY_{k-1})}.
\end{align*}
Applying the mean ergodic theorem in $\YY$ to  the function$f_\chi$, the right hand side of this equation converges to $0$ as $M\rightarrow\infty$. 
We conclude that whenever $\chi\not = 1$, 
$$\lim_{N\rightarrow\infty} \E_{n\in\IP_{\Phi_N}\bigl((n_j)_{j\in\mathbb{N}}\bigr)} f_\chi(T^n (y,t)) =0$$ for all $(y,t)\in Y$, completing the proof.
\end{proof}

We are set to prove Theorem~\ref{formula}.
\begin{proof}
With the preliminaries adapted to the setting of  $\IP$s, we now follow the argument in~\cite{BHK}. For every $x\in G/\Gamma$, let $g_x\in G$ be a lift of $x$ and set $\tilde{\Gamma}_x = \{(g_x\gamma_1g_x^{-1},\dots,g_x\gamma_kg_x^{-1}):(\gamma_1,\dots,\gamma_k)\in \tilde{\Gamma}\}$.  We consider the nilmanifold $\tilde{X}_x = \tilde{G}/\tilde{\Gamma}_x$  endowed with the Haar measure $\tilde{\mu}_x$. 
By~\cite[Corollary 5.5]{BHK}, the action of $T_\tau^{\ell_1}\times\dots\times T_\tau^{\ell_k}$ on $\tilde{X}_x$ is ergodic for $\mu$-almost every $x\in X$. 

Let $\YY$ be as in Theorem~\ref{IPpointwise} and let $f_1,\dots,f_k\colon X\rightarrow \mathbb{C}$ be the lifts of some continuous functions $g_1,\dots,g_k\colon Y\rightarrow \mathbb{C}$ (thus $f_i = g_i\circ\pi$ where $\pi\colon \XX\rightarrow \YY$ is an isomorphism). 
By Theorem~\ref{IPpointwise} applied to the point $(x,x,\dots,x)$, we deduce that 
    $$\lim_{N\rightarrow\infty} \sEIP \prod_{i=1}^k f_i(T_\tau^{\ell_in}x)$$ converges everywhere.  By ergodicity, it must converge to the integral with respect to the Haar measure $\tilde{\mu}_x$. 
    This holds for all $f_1,\dots,f_k$ that arise as such lifts, and since the continuous functions are dense in $L^\infty(\mu_Y)$, the lifts of such functions are dense in $L^\infty(\mu_X)$. Thus, we have convergence in $L^2(\mu_X)$ to the desired limit for all $f_1,\dots,f_k\in L^\infty(\mu_X)$.
\end{proof}
\subsection{Limit formulae for weighted averages}

In Theorems~\ref{Kroneckerformula} and~\ref{formula}, we assumed that the spectrum of $(n_j)_{j\in\mathbb{N}}$ contains no nontrivial eigenvalues.  To address settings where this does not hold, we study certain weighted averages.

\begin{lemma}[Characteristic factors for multiple weighted averages.]\label{twistedcharacteristic}
Let $\XX=(X,\mathcal{B},\mu,T)$ be an ergodic system,  $(n_j)_{j\in\mathbb{N}}$ be a sequence with rational spectrum, $k\geq 2$, $f_1,\dots,f_k\in L^\infty(\XX)$,  $\ell_1, \dots,\ell_k\in \mathbb{N}$, 
and let $a\in Z_1(\XX)$ denote the rotation induced on the Kronecker factor $Z_1(\XX)$  by $T$. 
If $\ell_1,\ell_2\in \mathbb{N}$ are coprime,    $\eta\colon Z_1(\XX)\rightarrow \mathbb{R}$ is a non-negative continuous function with $\EIP \eta(na)=1$, and  $\delta>0$, there exists a $(k-1)$-step nilsystem factor $\YY = (Y,\mathcal{Y},\nu,T)$ such that 
\begin{multline*}
    \big\|\EIP\eta(na) \prod_{i=1}^k f_i(T^{\ell_in}x) - \\ \EIP\eta(na)\prod_{i=1}^k \E(f_i\mid\Y)(T^{\ell_in}x)\big\|_{L^2(\mu)}<\delta 
    \end{multline*}
    for all increasing F\o lner sequences and sufficiently large $N$.
\end{lemma}
\begin{proof}
We start with the case that $\eta$ is a character. Since $\ell_1,\ell_2$ are coprime, we can choose  $s,t\in \mathbb{Z}$ such that $s\cdot \ell_1 +t\cdot \ell_2 = 1$. Letting $\pi\colon X\to Z_1(\XX)$ denote the factor map, set $\tilde{\eta} =\eta\circ \pi$ and write $f_1' = f_1 \cdot \tilde{\eta}^s$, $f_2' = f_2\cdot \tilde{\eta}^t$, and  $f_i'=f_i$ for all $i\geq 3$. Then $$\EIP\eta(na) \prod_{i=1}^k f_i(T^{\ell_in}x) = \EIP \prod_{i=1}^k f'_i(T^{\ell_in}x).$$
By  Theorem~\ref{HKcharacteristic}, this limit is zero if at least one function  $f_i'$ is orthogonal to $Z_{k-1}(\XX)$. Since $\eta$ is measurable with respect to $Z_1(\XX)$, the function $f_i'$ is orthogonal to the factor $Z_{k-1}(\XX)$ if and only if $f_i$ is orthogonal to $Z_1(\XX)$. 
For general $\eta$, using the Stone-Weierstrass Theorem we can approximate $\eta$ uniformly by linear combinations of characters of $Z_1(\XX)$. Thus, we may assume without loss of generality that all $f_i$ are all measurable with respect to $Z_{k-1}(\XX)$. 

Let $\delta>0$. Since the factor $Z_{k-1}(\XX)$ is an inverse limit of nilsystems~\cite{host2005nonconventional, ziegler2007universal}, there is a $(k-1)$-step nilsystem $\YY$ such that $\left\|f_i - \E(f_i\mid Y)\right\|<\delta/2k$. By the Cauchy-Schwartz and triangle inequalities, we have that
\[\left\|\prod_{i=1}^k f_i\circ T^{\ell_in}  - \prod_{i=1}^k \E(f_i\mid \Y)\circ T^{\ell_i n}\right\|_{L^2(\mu)}<\delta/2. 
\] Since $\lim_{N\rightarrow\infty}\sEIP \eta(na)=1$, it follows that for $N$ sufficiently large we have that
\begin{multline*}
    \big\|\sEIP \eta(na) \prod_{i=1}^k f_i(T^{\ell_i n}(x)) - \sEIP \eta(na) \prod_{i=1}^k \E(f_i\mid \Y)(T^{\ell_i n}(x))\big\|_{L^2(\mu)} \\ < \sEIP \eta(na)\cdot \delta/2\leq \delta. \quad\qedhere
    \end{multline*}
\end{proof}

We make use of various  results about nilmanifolds, and refer to~\cite{hk-book} for further details and background. For an ergodic nilsystem $\XX=(G/\Gamma,\mathcal{B},\mu,T_\tau)$, we let  $G^0$ denote the connected component of $G$ and set $\Gamma^0 = G^0\cap \Gamma$. 
The nilmanifold $G^0/\Gamma^0$ can be embedded as a subnilmanifold of  $G/\Gamma$, and since the quotient $G/G^0\Gamma$ is finite, there is some $r\in \mathbb{N}$ such that $\tau^r$ maps $G^0/\Gamma^0$ to itself and we have
$$G/\Gamma = \bigsqcup_{i=0}^{r-1} \tau^i G^0/\Gamma^0.$$
Let $\XX^0$ denote the associated system (endowed with the Borel $\sigma$-algebra, Haar measure $\mu^0$, and transformation $T_\tau^r = T_{\tau^r}$). 
 
\begin{definition}
    An ergodic nilsystem $\XX=(G/\Gamma,\mathcal{B},\mu,T_\tau)$ is \emph{synchronized on $G^0/\Gamma^0$}, or \emph{synchronized} if the context is clear, if there exist some $b\in G$ and $r\in\N$ such that
    \begin{enumerate}
        \item
        \label{item:goodnil-one}
        $T_b\colon G/\Gamma\rightarrow G/\Gamma$ maps $G^0/\Gamma^0$ to itself; 
        \item
         \label{item:goodnil-two}
        The action of $T_b$ on $G^0/\Gamma^0$ is ergodic; 
        \item
         \label{item:goodnil-three}
         The elements $b, \tau\in G$ satisfy 
         $b^r=\tau^r$.
    \end{enumerate}
\end{definition}
We note that for general nilsystems,  property~\eqref{item:goodnil-one} does not imply that $b\in G^0$. For instance, any totally ergodic  nilsystem  $\XX$ is synchronized, as then  $X=G^0/\Gamma^0$ and we can take $b=\tau$ and $r=1$. Moreover, all ergodic $1$-step nilsystems are synchronized, as seen by projecting the rotation $\tau$ of  a compact abelian Lie group $Z$ to the torus part of $Z$.  
We show that this extends to $2$-step nilsystems.  
\begin{lemma}\label{goodlemma}
Let $k=1$ or $k=2$. Then every ergodic $k$-step nilsystem $\XX=(G/\Gamma,\mathcal{B},\mu,T_\tau)$ is synchronized on $G^0/\Gamma^0$ for some choice of $G,\Gamma$.
\end{lemma}
The proof is technical and relies on results from~\cite{host2005nonconventional} not directly relevant to this paper, and so we defer it to Appendix~\ref{goodproof}.
While it is not needed for our current work, it would be interesting to know if all nilsystems are synchronized.

For  $r\in\mathbb{N}$, let 
$\one_{r\N}\colon \mathbb{N}\rightarrow \{0,1\}$ denote the function which assigns $1$ to any integer divisible by $r$ and $0$ otherwise. Given a F\o lner sequence $\Phi=(\Phi_N)_{N\in\mathbb{N}}$, set
$$\one_{r,\Phi}(n) = \frac{\one_{r\N}(n)}{\lim_{N\rightarrow\infty}\EIP \one_{r\N}(n)}.$$  This is well defined: the limit in the denominator exists because  $\one_{r\N}(n) = \frac{1}{r}\sum_{i=0}^{r-1}\xi_r^n$, where $\xi_r$ is the first root of unity of degree $r$, and so is a periodic function. 

\begin{proposition}
\label{twistedpointwise}
    Let $\XX=(G/\Gamma,\mathcal{B},\mu,T_\tau)$ be a nilsystem and suppose that $\XX$ is synchronized on $G^0/\Gamma^0$. Let $(n_j)_{j\in\mathbb{N}}$ be a sequence with rational spectrum and let $\Phi=(\Phi_N)_{N\in\mathbb{N}}$ be an increasing F\o lner sequence. Then
    there exists some $r\in\mathbb{N}$ and
   a  set of bounded functions $\CF$ that are dense in $L^2(\mu)$ such that  for all $f\in\CF$, we have $$\lim_{N\rightarrow\infty} \EIP \bm{1}_{r,\Phi}(n) f(\tau^nx) = \int_{X^0} f(xy) \d\mu^0(y\Gamma^0)$$ for $\mu_G$-almost every $x\in G$, where $X^0 = G^0/\Gamma^0$ is the connected component of the identity in $X$ and is endowed with the Haar measure $\mu^0$.
\end{proposition}
\begin{proof}
Since the system $\XX$ is synchronized, we can find some $r$ and $\tau\in G$ such that $\tau^r=b^r$ and $b$ maps the connected component of $X$ to itself. 
Since $\bm{1}_{r,\Phi}(n)\not=0$ if and only if $\tau^n=b^n$, 
it suffices to compute the limit
$$\lim_{N\rightarrow\infty}\EIP \bm{1}_{r,\Phi}(n) f(b^nx).$$

Decompose $X = \bigsqcup_{i=0}^{r'-1} X_i$ into connected components and note that by increasing $r$ if necessary we may assume without loss of generality that $r'$ divides $r$. The translation by $b$ induces an ergodic action on each $X_i$.  Thus the function $f$ can be written as a sum of functions on each of these components. 
Therefore, without loss of generality, it suffices to find a dense set of bounded functions $f\colon X^0\rightarrow \mathbb{C}$ such that
\begin{equation}
    \label{eq:dense-f}
    \lim_{N\rightarrow\infty}\sEIP \bm{1}_{r,\Phi}(n) f(b^nx)= \int_{X^0} f(x) \d\mu^0(x).
    \end{equation}

We proceed by induction on $k$. 
For $k=1$, the system $\XX^0=(X^0,\CB_{X^0}, \mu^0, R_b)$ is an ergodic rotation on a connected compact abelian group $Z$. 
Linear combinations of characters are dense in $C(Z)$, and so it suffices to prove the result for a character $f$.  
If $f=1$, then it follows from the definition of $\bm{1}_{r,\Phi}$ that the limit in~\eqref{eq:dense-f} is $1$. Otherwise, since $X^0$ is connected, $f(b^nx) = \alpha^n f(x)$ for some irrational $\alpha$. 
On the other hand $\bm{1}_{r,\Phi}(n) = c\cdot \sum_{i=0}^{r-1}\xi_r^{in} $ where $\xi_r$ is the first root of unity of degree $r$ and $c\in\mathbb{R}$ is a constant. 
We have
$$\lim_{N\rightarrow\infty}\sEIP \bm{1}_{r,\Phi}(n) f(b^nx) =c\sum_{i=0}^{r-1} \omega_{\Phi}(\xi_r^i\cdot \alpha)\cdot f(x)$$
where $\omega_{\Phi}$ is defined as in~\eqref{omega}. However, since $\xi_r^i \alpha$ is irrational, and $(n_j)_{j\in\mathbb{N}}$ has rational spectrum this quantity is zero, completing the case $k=1$. 

To proceed inductively, we use a similar argument to that in the proof of Theorem~\ref{IPpointwise}. Namely, the system $\XX^0$ is isomorphic to a system $\YY$ and $Y= Y_{k-1}\times_\rho U$ for some connected group $U$ and a continuous cocycle $\rho$.  We then consider a nilcharacter $f_{\chi}(y,t) = g(y)\chi(t)$, where $\chi\colon U\rightarrow \CS^1$ is a nontrivial character and $g\colon Y_{k-1}\rightarrow\mathbb{C}$ is continuous. Once again write $\bm{1}_{r,\Phi}(n)=c\cdot \sum_{i=0}^{r-1} \xi_r^{in}$, fix some $0\leq i \leq r-1$, and set $s=\xi_r^i$. 
Using the same notation as in the proof of Theorem~\ref{IPpointwise} (so $T_{k-1}$ denotes the action on $\YY_{k-1}$ and $T$ the action on $\YY$), we have 
\begin{align*}
\limsup_{N\rightarrow\infty} &\left|\sEIP   s^n\cdot f_\chi(T^n(y,t))\right|\\&= \limsup_{N\rightarrow\infty}\left|\sEIP s^n g(T_{k-1}^n y)\chi\circ\rho(n,y) f_\chi(y,t)\right|.
\end{align*}
By the van der Corput inequality (Lemma~\ref{vdc}) and the induction hypothesis, this is bounded by
$$\limsup_{M\rightarrow\infty}\|\E_{m\in\IP_{\Phi_M}} s^mg(T_{k-1}^my)\rho(m,y)\|_{L^2(\mu_{k-1})}^2 =0$$
Since $X^0$ is connected, it is totally ergodic and so the spectrum is supported on the irrational elements in $\CS^1$. Therefore, we have the $L^2$-convergence of the average $\lim_{N\rightarrow\infty}\sEIP s^n f_{\chi}(T^n(y,t))=0$ for all rational $s\in \CS^1$Th, completing the proof. 
\end{proof}

As a corollary, we obtain a formula for weighted averages on synchronized nilsystems. 
\begin{corollary}
\label{twistedformulanilsystem}
Let $\XX=(G/\Gamma,\mathcal{B},\mu,T_\tau)$ be an ergodic $k$-step nilsystem that is synchronized on $G^0/\Gamma^0$. Let $\eta\colon Z_1(\XX)\rightarrow \mathbb{R}$ be a non-negative continuous function such that  $\lim_{N\rightarrow\infty}\EIP \eta(na)=1$ where $a$ denotes the rotation induced on the Kronecker factor $Z_1(\XX)$, let $f_1,\dots,f_k\in L^\infty(\mu)$ be bounded functions, and let $\ell_1,\dots,\ell_k\in\mathbb{N}$ be distinct integers such that $\ell_1,\ell_2$ are coprime. Then there exists $r\in\N$ such that for every increasing F\o lner sequence $\Phi=(\Phi_N)_{N\in\mathbb{N}}$, we have
    \begin{multline*}
        \lim_{N\rightarrow\infty}\EIP \eta(na) \bm{1}_{r,\Phi}(n) \prod_{i=1}^k f_i(T^{\ell_in}x)  \\  = \int_{\tilde{X}^0} \eta(y_1)\prod_{i=1}^k f_i(xy_i)~\d\tilde{\mu}(y\tilde{\Gamma}_{\ell_1,\dots,\ell_k}), \end{multline*}
where $\tilde{X}^0$ denotes the connected component of $\tilde{X}$ and $\tilde{\mu}^0$ the Haar measure on $\tilde{X}^0$.
\end{corollary}

\begin{proof}
As in the proof of Lemma~\ref{twistedcharacteristic}, we can absorb $\eta$ into $f_1,f_2$ and so without loss of generality  we assume that $\eta=1$. As in the proof of Proposition~\ref{twistedpointwise}, we can reduce to the case that the functions are defined on the connected component of the identity (note that all of these component are isomorphic). Using the proof of Theorem~\ref{formula}, the limit is obtained by applying a pointwise convergence result for an ergodic average of a single function on a particular nilsystem. We  therefore obtain the required formula using the same proof as in Theorem~\ref{formula}, only replacing the use of Theorem~\ref{IPpointwise} by Proposition~\ref{twistedpointwise}.
\end{proof}

\section{The large intersection property}\label{large:sec}

\subsection{Notions of largeness}
Depending on the choice of sequence, different notions of size are relevant. 
\begin{definition}[IP-density]\label{IP-density:def}
    Let $(n_j)_{j\in\mathbb{N}}$ be a sequence of natural numbers, let $\Phi=(\Phi_N)_{N\in\mathbb{N}}$ be an increasing sequence, and let $A\subseteq \IP((n_j)_{j\in\mathbb{N}})$ be a subset. We define the \emph{lower $\IP$-density} of $A$ with respect to $\Phi$ as the quantity
    $$\underline{d}_{\IP_{\Phi}}(A):=\liminf_{N\rightarrow\infty} \frac{|A\cap \IP_{\Phi_N}\bigl((n_j)_{j\in\mathbb{N}}\bigr)|}{|\IP_{\Phi_N}\bigl((n_j)_{j\in\mathbb{N}}\bigr)|}.$$
\end{definition}
Other notions of largeness for subsets of $\mathbb{N}$ also have counterparts for $\IP$s, and we include a discussion in Appendix~\ref{largenotions}.

\subsection{Lower bounds for three terms}

\begin{theorem}\label{three}
    Let $\XX=(X,\mathcal{B},\mu,T)$ be an ergodic invertible measure preserving system, let $(n_j)_{j\in\mathbb{N}}$ be a sequence with rational spectrum, and let $A\subseteq X$ be a measurable set. Then for all $\varepsilon>0$,  increasing F\o lner sequence $\Phi=(\Phi_N)_{N\in\mathbb{N}}$, and  coprime and nonzero integers $\ell_1,\ell_2\in \mathbb{Z}$, we have
    $$\dip\{n\in \IP\big((n_j)_{j\in\mathbb{N}}\big) : \mu(A\cap T^{-\ell_1n}A\cap T^{-\ell_2n}A)>\mu(A)^3-\varepsilon\}>0.$$
\end{theorem}
\begin{proof}
The factor $Z_1(\XX)$ is (measurably isomorphic to) a rotation on a compact abelian group and we let $a\in Z_1(\XX)$ denote the rotation. Let $\eta\colon Z_1(\XX)\rightarrow\mathbb{R}$ be a non-negative continuous function (yet to be specified)  such that $\lim_{N\rightarrow\infty} \sEIP \eta(na)=1$. For contradiction, suppose that the statement is false for some $0<\varepsilon<1$.  Then there exists an increasing F\o lner sequence $\Phi=(\Phi_N)_{N\in\mathbb{N}}$ so that 
\begin{equation}\label{few}
    \lim_{N\rightarrow\infty}\frac{|R_{\varepsilon,N}|}{|\IP_{\Phi_N}\bigl((n_j)_{j\in \N}\bigr)|}=1,
\end{equation}
where $$R_{\varepsilon,N}=\{n\in\IP_{\Phi_N}\bigl((n_j)_{j\in \N}\bigr) : \mu(A\cap T^{-\ell_1n}A\cap T^{-\ell_2n}A)\leq \mu(A)^3-\varepsilon\}.$$

Since $\eta$ is continuous and $Z_1(\XX)$ is compact, 
it is bounded by a constant (not depending on $N$), and so we have
\begin{align*}\sEIP \eta(na)\mu(A\cap T^{-\ell_1n}A\cap& T^{-\ell_2n}A) \\&\leq  \frac{1}{|\IP_{\Phi_N}|}\sum_{n\in R_{\varepsilon,N}}\eta(na)(\mu(A)^3-\varepsilon) + \|\eta\|_\infty\cdot o_{N\rightarrow\infty}(1)\\&\leq \sEIP \eta(na) \cdot (\mu(A)^3-\varepsilon) + \|\eta\|_\infty \cdot o_{N\rightarrow\infty}(1).
\end{align*}

Passing to a subsequence of $\Phi_N$ if necessary and taking $N\rightarrow\infty$, we conclude that
\begin{equation}\label{threeupper}
\lim_{N\rightarrow\infty}\sEIP \eta(na)\mu(A\cap T^{-\ell_1n}A\cap T^{-\ell_2n}A) <\mu(A)^3-\varepsilon.
\end{equation}
Choosing $\delta>0$ sufficiently small with respect to $\varepsilon$, it follows from Lemma~\ref{twistedcharacteristic} that there exists a $1$-step nilsystem factor $\YY$ such that if $f$ is the projection of $\one_A$ onto $\YY$, then 
\begin{multline}\label{close}
  \big| \lim_{N\rightarrow\infty} \Big(\sEIP \eta(na)\mu(A\cap T^{-\ell_1n}A\cap T^{-\ell_2n}A) \\ - \sEIP \eta(na)\int_\YY f(x) f(T^{\ell_1 n}x) f(T^{\ell_2 n}x) ~\d\mu_{\YY}(x)\Big)\big| < \varepsilon/2.
\end{multline}
As a $1$-step nilsystem is a rotation on a compact abelian Lie group, we can write $Y=\T^m\times D$ for some finite dimensional torus $\T^m= (\R/\Z)^m$ for some $m\geq 1$ and  finite cyclic group $D$, 
with the rotation by the element $a=(a_1,a_2)$ for some $a_1\in \T^m$ and $a_2\in D$. 
Let $\eta\colon Y\rightarrow \mathbb{R}$ be a function of the form $\eta(t,d) = \eta_1(t)\cdot c^{-1} \one_{\{0\}}(d)$, where $\eta_1\colon\T^m\rightarrow \mathbb{R}$ is a non-negative continuous function still to be chosen, $\one_{\{0\}}(d)$ takes the value $1$ if $d=0$ and the value $0$ otherwise, and $c := \lim_{N\rightarrow\infty} \sEIP \one_{\{0\}}(na_2)$. Then equation~\eqref{close} holds for this $\eta$. 
Note that $\one_{\{0\}}(na_2) = \one_{r}(n)$ where $r=|D|$ and so $c$ is nonzero and $c^{-1}\cdot \one_{\{0\}}(na_2) = \one_{r,\Phi}(n).$ Lifting 
$\eta$ to the Kronecker factor of $\XX$, by equation~\eqref{close} and Corollary~\ref{twistedformulanilsystem}, we have that the left hand side of~\eqref{threeupper} is $\varepsilon/2$-close to the  quantity
$$\int_\YY\int_{\T} \eta(t) f(x)f(x+\ell_1 t) f(x+\ell_2 t) \d\mu_{\T}(t)\d_{\mu_{\YY}}(x).$$

Let $V$ be small neighborhood of identity in $\T^m$ and assume that $\eta_1$ is a continuous function supported on $V$ such that $\int_{\T^m} \eta_1~d\mu=1$. 
Since the action of $a_1$ on $\T^m$ is totally ergodic, we have that 
\[\lim_{N\rightarrow\infty} \sEIP \eta_1(na_1)\cdot c^{-1} \one_{\{0\}}(na_2) =\int_{\T} \eta(t)~dt = 1.\]
It follows that by taking the neighborhood $V$ to be sufficiently small, the left hand side of  Equation~\eqref{threeupper} is $\varepsilon$-close to
$$\int_Z f(x)f(x)f(x) \d\mu_Z(x) \geq \left(\int_Z f(x) \d\mu_Z(x)\right)^3 = \mu(A)^3.$$ In other words,
$$\lim_{N\rightarrow\infty}\sEIP \eta(na)\mu(A\cap T^{-\ell_1n}A\cap T^{-\ell_2n}A) > \mu(A)^3-\varepsilon,$$
 a contradiction of the inequality in~\eqref{threeupper}. 
\end{proof}

\subsection{Four terms}
 Before extending the lower bounds to four terms, we derive a result that follows from Frantzikinakis~\cite[Proof of Theorem C]{Franpoly}.
\begin{lemma}\label{Fran}
    Let $\XX=(G/\Gamma,\mathcal{B},\mu,T)$ be an ergodic $2$-step nilsystem, $f\in L^\infty(\mu)$ be a bounded function,   $\ell_1,\ell_2$ be coprime and nonzero integers, and  set $\ell_0=0$ and  $\ell_3=\ell_1+\ell_2$. Then
    $$\int_{X}\int_{G_2}\int_{G_2} \prod_{i=0}^3 f(xy_1^{\ell_i}y_2^{\binom{\ell_i}{2}})\d\mu_{G_2}(y_2)\d\mu_{G_2}(y_1)\d\mu(x) \geq \left(\int_X f \d\mu\right)^4.$$
\end{lemma}
\begin{proof}
Since $G_2$ commutes with all elements in $G$, we can rewrite the integral on the left hand side as 
$$\int_X \int_{G_2\times G_2} \prod_{i=0}^3 f(y_1^{\ell_i}y_2^{\binom{\ell_i}{2}}x) \d\mu_{G_2\times G_2}(y_1,y_2)\d\mu_X(x).$$
Taking the integral over $G_2$, this becomes 
\begin{equation}\label{G2}
\int_X \int_{G_2\times G_2\times G_2} \prod_{i=0}^3 f(yy_1^{\ell_i}y_2^{\binom{\ell_i}{2}}x) \d\mu_{G_2\times G_2\times G_2}(y,y_1,y_2)\d\mu_X(x).
\end{equation}
Reparameterizing the set $$\{(y,yy_1^{\ell_1}y_2^{\binom{\ell_1}{2}},yy_1^{\ell_2}y_2^{\binom{\ell_2}{2}},yy_1^{\ell_1+\ell_2}y_2^{\binom{\ell_1+\ell_2}{2}} : y,y_1,y_2\in G_2\}$$ as 
$$\{(h_1,h_2,h_3,h_4)\in G_2^4 : h_1^{\ell_2-\ell_1}h_3^{\ell_1+\ell_2} = h_4^{\ell_2-\ell_1}h_2^{\ell_1+\ell_2}\},$$ we can rewrite~\eqref{G2} as 
$$\int_\XX \int_{G_2} \left(\int_{\{h_1^{\ell_2-\ell_1}h_3^{\ell_1+\ell_2}=h\}} f(h_1x)f(h_3x)\d\mu_{G_2\times G_2}(h_1,h_3)\right)^2\d\mu_{G_2}(h)\d\mu_X(x).$$
By the Cauchy-Schwartz Inequality and a change of variables, this last expression is greater than  or equal to 
$$\int_\XX \left(\int_{G_2} f(hx)\d\mu_{G_2}(h)\right)^4 \d\mu_X(x). $$
This, in turn, is greater or equal than
$$\left(\int_\XX \int_{G_2} f(hx)\d\mu_{G_2}(h)\d\mu_X(x)\right)^4 = \left(\int_\XX f(x)\d\mu_X(x)\right)^4, $$ as required.
\end{proof}

\begin{theorem}\label{four}
Let $\XX=(X,\mathcal{B},\mu,T)$ be an ergodic invertible measure preserving system, let $(n_j)_{j\in\mathbb{N}}$ be a sequence with rational spectrum, and  let $A\subseteq \XX$ be a measurable set.  Then for all $\varepsilon>0$, increasing F\o lner sequence $\Phi=(\Phi_N)_{N\in\mathbb{N}}$, and  coprime and nonzero integers $\ell_1,\ell_2\in\mathbb{Z}$, we have
$$\dip\{n\in \IP\bigl((n_j)_{j\in\mathbb{N}}\bigr) : \mu(A\cap T^{-\ell_1n}A\cap T^{-\ell_2n}A\cap T^{-(\ell_1+\ell_2)n}A)>\mu(A)^4-\varepsilon\}>0.$$
\end{theorem}

\begin{proof}
We start as in the proof of Theorem~\ref{three}, identifying $Z_1(\XX)$ with a compact abelian group and letting $a\in Z_1(\XX)$ denote the rotation. 
Again let $\eta\colon Z_1(\XX)\rightarrow\mathbb{R}$ be a continuous non-negative function such that $$\lim_{N\rightarrow\infty} \sEIP \eta(na)=1$$ and suppose by contradiction that the statement does not hold. Then, as in the proof of Theorem \ref{three}, for can find some $0<\varepsilon<1$ such that 
\begin{equation}\label{upperfour}
    \lim_{N\rightarrow\infty} \sEIP \eta(na)\mu(A\cap T^{-\ell_1n}A\cap T^{-\ell_2n}A\cap T^{-(\ell_1+\ell_2)n}) \leq \mu(A)^4 - \varepsilon.
\end{equation}
Taking $\delta>0$ sufficiently small with respect to $\varepsilon>0$, it follows from  Lemma~\ref{twistedcharacteristic} that there exists a factor isomorphic to an ergodic $2$-step nilsystem $\YY$ such that letting $f$ denote the projection of $\one_A$ onto this factor, we have that the left hand side of~\eqref{upperfour} is $\varepsilon/2$-close to
\begin{equation}\label{close4}
      \lim_{N\rightarrow\infty} \sEIP \eta(na)\int_{Y} f(x)f(T^{\ell_1n}x) f(T^{\ell_2n}x)f(T^{(\ell_1+\ell_2)n}x)~\d\mu(x).
\end{equation}

By Lemma~\ref{goodlemma}, the system $\YY$ is synchronized. Let $r$ be as in Lemma~\ref{goodlemma} and let $Z(\YY)$ denote the Kronecker factor of $\YY$. Again write $Z=\T^m\times D$ for some finite dimensional torus $\T^m$ and a finite cyclic group $D$, and we write the translation by $T$ as $a=(a_1,a_2)\in Z(\YY)$ with $a_1\in \T^m$ and $a_2\in D$. For convenience, we assume that $r$ divides $|D|$, and otherwise we can pass to an extension such that this holds.  
Take $\eta\colon Z(\YY)\rightarrow \mathbb{R}$ to be in the proof of Theorem~\ref{three}, meaning that it is a function the form $\eta(t,d) = \eta_1(t)\cdot c^{-1}\one_{\{0\}}(d) $ where $\eta_1\colon  \T^m\rightarrow \mathbb{R}$ is a non-negative continuous function yet to be specified and $c:=\lim_{N\rightarrow\infty}\sEIP \one_{\{0\}}(na_2)$, and lift this function to the Kronecker factor of $\XX$. 
Let $Y^0=N/\Lambda$ denote the connected component of $Y$. 
For notational simplicity, set  $\ell_0=0$ and $\ell_3=\ell_1+\ell_2$. Then, by~\eqref{close4} and Proposition~\ref{twistedformulanilsystem}, the left hand side of~\eqref{upperfour} is $\varepsilon/2$ close to
\begin{equation}\label{firstformula}
    \int_{\YY} \int_{N/\Lambda}\int_{N_2/\Lambda_2} \eta_1(y_1\Lambda) \prod_{i=0}^3 f(xy_1^{\ell_i}y_2^{\binom{\ell_i}{2}})~ d\mu_{N_2/\Gamma_2}(y_2) \d\mu_{N/\Gamma}(y_1)\d\mu_{\YY}(x).
\end{equation}
Let $V=B(G_2,r)$ be an open ball with radius $r>0$ and choose $r$  to be sufficiently small with respect to $\varepsilon>0$ and such that $V$ has trivial intersection with the group $D$.  Let  $\eta_1\colon\T^m\rightarrow \mathbb{R}$ be a function supported on the projection of $V$ onto $\T^m$  and such that the lift to $Z$ satisfies $\int_{\T}\eta_1(t)~d\mu(t)=1$. Since the action of $a_1$ on $\T^m$ is totally ergodic, 
it follows that 
\[\lim_{N\rightarrow\infty}\sEIP \eta_1(na_1)\cdot c^{-1}\one_{\{0\}}(d) = \int_{\T} \eta_1(t) ~\d\mu_{\T}(t) = 1.
\] 
Using~\eqref{firstformula}, it follows that the left hand side of equation~\eqref{upperfour} is $\varepsilon$-close to
$$\int_{Y} \int_{N_2/\Lambda_2}\int_{N_2/\Lambda_2} \prod_{i=0}^3 f(xy_1^{\ell_i}y_2^{\binom{\ell_i}{2}}) \d\mu_{N_2/\Lambda_2}(y_1,y_2) \d\mu_{Y}(x).$$
By Lemma~\ref{Fran}, this is bounded below by $\mu(A)^4$, a contradiction of~\eqref{upperfour}.
\end{proof}

\appendix
\section{On different notions of largeness}\label{largenotions} 

There are several other notions of largeness that are appropriate for $\IP$-sequences, and we discuss relations among some of them.  Let $(n_j)_{j\in\mathbb{N}}$ be a sequence of natural numbers and let $A\subseteq \IP\bigl((n_j)_{j\in\mathbb{N}}\bigr)$.
    \begin{enumerate}
        \item \label{item:densityone}
        The set $A$ has \emph{positive lower $\IP$-density} if for any increasing F\o lner sequence $(\Phi_N)_{N\in\mathbb{N}}$,  we have $\dip(A)>0.$ Replacing the $\liminf$ in the definition of $\dip$ with the limit or with $\limsup$, we define \emph{positive $\IP$-density} and \emph{positive upper $\IP$-density} of the set $A$. 
        \item \label{item:densitytwo}
        The set $A$ is \emph{syndetic in $\IP\bigl((n_j)_{j\in\mathbb{N}}\bigr)$} if finitely many translations of $A$ cover $\IP\bigl((n_j)_{j\in\mathbb{N}}\bigr)$. 
        \item 
        \label{item:densitythree}
        The set $A$ is \emph{$\IP$-syndetic} if it has nontrivial intersection with every $\IP$-F\o lner sequence.
    \end{enumerate}

We check that none of these density notions are equivalent.
\begin{lemma}
    For each of the density conditions~\eqref{item:densityone}, \eqref{item:densitytwo}, and~\eqref{item:densitythree}, there exists a set satisfying that condition but not the other two.
\end{lemma}
\begin{proof}
Let $(n_j)_{j\in\mathbb{N}}$ be a sequence such that  all of its finite sums are distinct (for example, we can take  $n_j=10^{j-1}$).

To construct a set that has 
positive $\IP$-density but is neither syndetic in  
$\IP\bigl((n_j)_{j\in\mathbb{N}}\bigr)$ nor is $\IP$-syndetic, set $\Phi_N = [N^2,N^2+N]$, let $\Phi = \bigcup_{N\in\mathbb{N}}\Phi_N$, and take $A = \IP\bigl((n_j)_{j\in\mathbb{N}}\bigr)\backslash \IP_{\Phi}((n_j)_{j\in\mathbb{N}})$. The set $A$ is not $\IP$-syndetic because it has trivial intersection with  $\IP_{\Phi_N}\bigl((n_j)_{j\in\mathbb{N}}\bigr)$ and is not syndetic as infinitely many translates are needed to cover $\IP\bigl((n_j)_{j\in\mathbb{N}}\bigr)$.  By construction of the sequence $\Phi_N$, the set $A$ has positive $\IP$-density. 

The set $A = \{n_j : j\in\mathbb{N}\}$ is easily checked to be $\IP$-syndetic, but does not have positive lower $\IP$-density and is not syndetic in $\IP\bigl((n_j)_{j\in\mathbb{N}}\bigr)$. 

Taking the set $A\subseteq \IP\bigl((n_j)_{j\in\mathbb{N}}\bigr)$ to be the sequence which contains all summands which include $n_1$, then we have that this set avoids the $\IP$-F\o lner sequence $\Phi_N = [2,N]$ and is therefore not $\IP$-large nor $\IP$-syndetic.  However, since  $A\cup A-n_1 = \IP\bigl((n_j)_{j\in\mathbb{N}}\bigr)$, this set is syndetic. 
\end{proof}
We note that it follows immediately from the definitions that for any sequence  $(n_j)_{j\in\mathbb{N}}$ of natural numbers, if a set $A$ intersects every $\IP$ set, then $A\cap \IP\bigl((n_j)_{j\in\mathbb{N}}\bigr)$ is $\IP$-syndetic in $\IP\bigl((n_j)_{j\in\mathbb{N}}\bigr)$.

\section{Failure of the large intersection property for non-rational sequences}\label{failure:sec}
We show that Theorem~\ref{three} does not hold for arbitrary sequences, even when the notion of $\IP$-largeness is replaced by some other notions of largeness, including stronger ones. 
A set of integers is \emph{central} if it is an element in a minimal idempotent in the Stone-\v{C}ech compactification  $\beta\N$ of the integers (see~\cite{furstenberg2014recurrence} for more on central sets and~\cite{Bergelsonidempotents} for this equivalent definition).

It was established in~\cite{Bergelsonidempotents} that a central set $I$ contains all finite sums of a sequences $(n_j)_{j\in\mathbb{N}}$ however, without multiplicities. Throughout we  assume that all central sets contains $0$ by default. 
Though our notion of $\IP$ differs from that used in~\cite{furstenberg2014recurrence}, given $\IP$ generated by a sequence of distinct elements $(n_j)_{j\in\mathbb{N}}$, we  can always find a sub-$\IP$ by passing to a subsequence $(n_{j_i})_{i\in\mathbb{N}}$ satisfying that $n_{j_i}>\sum_{t=1}^{i-1}n_{j_t}$ for all $i>1$. 
Thus the $\IP$ generated by this subsequence is a set, rather than a multiset, and   a central set contains an $\IP$.

\begin{theorem}\label{false}
There exists an ergodic invertible measure preserving system $\XX=(X,\mathcal{B},\mu,T)$, a constant $c>0$, and a set $A\subseteq X$ of positive measure such that  any central set $I$ contains a central subset $I'\subseteq I$ with the property that 
$$\{0\not = n\in I' : \mu(A\cap T^{-n} A\cap T^{-2n}A)\geq \mu(A)^{-c\log \mu(A)}\} = \emptyset.$$
\end{theorem}

\begin{proof}
Consider the space $X = \T^2$, where $\T = \R/\Z$,  equipped with the action $T(x,y) = (\alpha+x, y+x)$ for some irrational $\alpha\in \T$. Since $\alpha$ is irrational, this system is ergodic and it is easy to check that the iterates of a point $(x,y)$ are given by
$$T^n(x,y) = (n\alpha+x, y+nx+\binom{n}{2}\alpha).$$

Set $A=\T\times B$ for a subset $B$ yet to be chosen.
Writing $\one_A(x,y) = \sum_{n\in\mathbb{Z}} a_n y^n$ in terms of the Fourier series for $\one_A$ (note that $\one_A$ does not depend on the first coordinate), we have that 
\begin{align*}
    \mu(A\cap & T^{-n}A\cap T^{-2n}A) = \int_{\XX} \one_A\cdot T^n \one_A\cdot T^{2n}\one_A \d\mu\\&=\sum_{m_0,m_1,m_2\in\mathbb{Z}} a_{m_0}a_{m_1}a_{m_2} \alpha^{m_1\binom{n}{2}+m_2\binom{2n}{2}} \int_{X} y^{m_0+m_1+m_2} x^{n(m_1+2m_2)}\d\mu(x,y).
\end{align*}
The last integral is zero unless $m_0+m_1+m_2=0$ and $m_1+2m_2=0$ (note that $n\not =0$). In this case, the integral is $1$. It follows that $m_0=m_1=-2m_2$, and so 
\begin{equation}\label{intersections}
\mu(A\cap T^{-n}A\cap T^{-2n}A) = \sum_{m\in\mathbb{Z}} a_m^2 a_{2m} \alpha^{mn}.
\end{equation}

Let $I$ be a central set and let $p$ be a minimal idempotent containing $I$. For an ultrafilter  $p\in\beta\N$, we write   $p\text{-}\lim_{n\in\mathbb{N}} x_n = x$ to mean that for every neighborhood $U$ of $x$ we have $\{n: x_n\in U\}\in p$
.  Bergelson~\cite{bergelson2003minimal}  showed that $p\text{-}\lim_{n\in\mathbb{N}} T^n = P_K$ converges weakly to the projection onto the Kronecker factor.  In particular, we have that $p\text{-}\lim_{n\in\mathbb{N}} (\alpha^{mn}) = 1$ for all $m\in\Z$. Therefore, applying $p\text{-}\lim$ to~\eqref{intersections},  we deduce that
\begin{equation}
\begin{split}
\label{eq:plim}
p\text{-}\lim_{n\in\mathbb{N}} \mu(A\cap & T^{-n}A\cap T^{-2n}A) = \sum_{m\in\mathbb{Z}} a_m^2 a_{2m} \\&=\int_{\mathbb{R}/\mathbb{Z}\times \mathbb{R}/\mathbb{Z}} \one_{B}(t) \one_{B}(t+s)\one_{B}(t+2s)\d s\d t.
\end{split}
\end{equation}
In~\cite[Proof of Theorem 2.1]{BHK}, it is shown that for every integer $L\geq 1$ there is a set $B_L$ of measure $\frac{\exp(-c\sqrt{\log L})}{4}$, where $c$ is the constant from Behrend's construction of a set of of integers of size $L$ containing no three term arithmetic progressions, for which the quantity in~\eqref{eq:plim} is bounded by $\frac{\exp(-c\sqrt{\log L})}{16L}.$ Taking $L$ sufficiently large, this implies that 
\[
p\text{-}\lim_{n\in\mathbb{N}} \mu(A\cap T^{-n}A\cap T^{-2n}A)  < \mu(A)^{-c\log \mu(A)}. 
\]
It follows that there exists a central set $I'\subseteq I$ such that $\mu(A\cap T^{-n}A\cap T^{-2n}A)  \leq \mu(A)^{-c\log \mu(A)}$ for all $n\in I'$, completing the proof.
\end{proof}

\section{Proof of Lemma~\ref{goodlemma}}\label{goodproof}
To prove Lemma~\ref{goodlemma}, we freely make use of notation and results from~\cite{host2005nonconventional}, giving precise references  but omitting the details.

We start with the following structure theorem for cocycles of type $2$ (see~\cite[Section 7]{host2005nonconventional} for definitions).
\begin{lemma}\label{type2classification}
Assume that $U_1$ is a finite dimensional torus, $K$ is a finite cyclic group, and let $Z=U_1\times K$ be a $1$-step nilsystem with the rotation given by some $a=(a_1,b)\in Z$ for some generator $b\in K$ and irrational $a_1\in U_1$. 
For a finite dimensional torus $U$, any cocycle $\rho\colon Z\rightarrow U$ of type $2$ is cohomologous to a sum of a $K$-invariant cocycle $\rho^0$ and a homomorphism $\phi\colon K\rightarrow U$.
\end{lemma}

\begin{proof}
Set $r=|K|$. For $z\in Z$, write $z=(x,t)$ for $x\in U_1$ and $t\in K$. For $s\in Z$ and $f\colon Z\rightarrow U$, we write $\partial_s f(x) = f(x+s)-f(x).$ 
    By~\cite[Lemma 8.1]{host2005nonconventional},
    we have that 
    $$\partial_b\rho = c-\partial_a F$$ for some $c\in U$ and measurable function $F\colon Z\rightarrow U$.
    Combining this equation with the telescoping identity $\sum_{i=0}^{r-1} \partial_b \rho(x,t+ib) =0 $, it follows that 
    \[r\cdot c = \partial_a \Bigl(\sum_{i=0}^{r-1} F(x,t+ib)\Bigr).
    \]
    Thus $rc$ is an eigenvalue of a $b$-invariant eigenfunction. 
    Since any eigenfunction is of the form $\psi\colon U_1\rightarrow U$ for some affine map $\psi$ (meaning that $\psi(x)-\psi(0)$ is a homomorphism from $U_1$ to $U$), we can write
    \begin{equation}\label{eq}\psi(x) = \sum_{i=0}^{r-1} F(x,t+ib).
    \end{equation}
    
    We claim that  $\psi$ determines $F$ up to a coboundary with respect to the action of $K$. More precisely, the claim is that 
     $$F(x,t) = \partial_b F'(x,t) + \binom{|t|}{r-1}\psi(x)$$ for some measurable map $F'\colon Z\rightarrow U$, where $\binom{|t|}{r-1}$ maps all cosets of $r\mathbb{Z}$ to $0$ other than the coset $(r-1)+r\mathbb{Z}$, which is mapped to $1$.

    To check the claim, note that $\sum_{i=0}^{r-1} \binom{|t+ib|}{r-1}\psi(x) = \psi(x)$ and so by~\eqref{eq} it follows that 
    \begin{equation} \label{torsionequation}
        \sum_{i=0}^{r-1} F(x,t+ib) - \binom{|t+ib|}{r-1}\psi(x) = 0.
    \end{equation}
   Defining $F'(x,jb) = \sum_{i=0}^{j-1} F(x,0+jb)- \binom{|jb|}{r-1}\psi(x)$ for all $0\leq j \leq r-1$, we have that~\eqref{torsionequation} implies  that $\partial_b F'(x,t) = F(x,t) - \binom{|t|}{r-1}\psi(x)$, proving the claim.

Thus by the claim, we may modify $\rho$ by the cohomologous cocycle $\rho-\partial_a F'$. Therefore, we can assume without loss of generality that $F(x,t) = \binom{|t|}{r-1}\psi(x)$ for some affine map $\psi\colon Z\rightarrow U$. In this case
    $$\partial_b \rho = c-\partial_b\left(\binom{|t|}{r-1}\psi(x)\right)=c - \binom{|t+b|}{r-1}\psi(x+a_1) + \binom{|t|}{r-1}\psi(x)$$ and $rc = \partial_{a_1} \psi(x)$. This equation determines $\rho$ up to a $b$-invariant function. Indeed, let $\sigma\colon Z\rightarrow U$ be the cocycle defined by 
   \[\sigma(x,jb) := \partial_{jb} \rho(x,0)=j\cdot c - \binom{|jb|}{r-1}\psi(x+a_1).
   \]
   Then $\partial_b \sigma(x,t) = \partial_b \rho(x,t)$ and so there is a map $f\colon U_1\rightarrow U$ such  that $\rho(x,t) = \sigma(x,t) + f(x)$.
   Since $\rho(x,t)$ is of type $2$, for every $s\in U_1$ 
   there is a constant $c_s\in U$ and a measurable map $F_s\colon Z\rightarrow U$ such that
\[\partial_s \sigma(x,t) + \partial_s f(x) = \partial_s \rho(x,t) = c_s - \partial_a F_s(x,t).\]
   When $t=0$, we have that $\sigma(x,0)=0$ and so $$\partial_s f(x) = c_s - \partial_a F_s(x,t.)$$
Choosing $t_0$ such that $t_0b=r-1$, we have $\sigma(x,t_0b)=t_0\cdot c-\psi(x+a_1)$. From this it follows that,
\[\partial_s \psi(x+a_1) +\partial_s f(x) = c_s - \partial_a F_s(x,t_0).\]
It follows that $\partial_s \psi(x+a_1)=\partial_s \psi(x)$ is a coboundary for all $s\in U_1$, but this can not happen unless $\psi = 0$. 

In this case,  $\partial_b \rho = c$, and again 
using the telescoping identity $\sum_{i=0}^{r-1} \partial_b \rho(x,t+ib) =0 $, we deduce  that $rc=0$. Hence, we can define a homomorphism $\phi\colon K\rightarrow U$ taking $b$ to $c$,  and then $\rho^0:=\rho-\phi$ is a $b$-invariant cocycle and $\rho$ is cohomologous to $\rho^0+\phi$.
\end{proof}

We use this to complete the proof of the lemma. 
\begin{proof}[Proof of Lemma~\ref{goodlemma}]
Let $k=1$ or $k=2$ and let $\XX=(G/\Gamma,\mathcal{B},\mu,T_\tau)$ be an ergodic $k$-step  nilsystem. 

Suppose first that $k=1$. Then $X$ is isomorphic to a rotation on a compact abelian Lie group $Z$ and we can write $Z=U_1\times K$ for some torus $U_1$ and  finite abelian cyclic group $K$, with 
the rotation $a=(a_1,b)$, where $a_1\in U_1$ and $b\in K$. Taking $r=|K|$, we have that $ra = r(a_1,0)$ and so $X$ is synchronized.

Now suppose that $k=2$, and write $X = \mathcal{G}(\XX)/\Gamma$ where $\mathcal{G}(\XX)$ is the Host-Kra group of $\XX$ (see~\cite[Section 5]{host2005nonconventional}). By~\cite[Section 8]{host2005nonconventional}, we have that $X = Z\times_\rho U$ where $Z$, $U$, and $\rho$ are as in Lemma~\ref{type2classification}. Thus, without loss of generality we may assume that $\rho = \rho^0 + \phi$, where $\rho^0$ is a $K$-invariant cocycle for a finite cyclic group $K$ and $\phi\colon K\to U$ is a homomorphism. Since $K$ has some finite order, it follows from the cocycle identity that there exists some $r\in \mathbb{N}$ such that $\rho(r,x) = \rho^0(r,x)$. The translation $T_\tau$ on $\XX$ is given by the element $\tau=S_{(a,\rho)}\in \mathcal{G}(\XX)$. On the other hand we also have $b:=S_{(a_1,0,\rho^0)}\in \mathcal{G}(\XX)$ since $\partial_{a_1}\rho  = \partial_a \rho^0$ (see~\cite[Lemma 8.7]{host2005nonconventional}) and the latter maps the connected component of $\XX$ to itself. Furthermore, it follows from the construction that $b^r=S_{((a_1,0),\rho^0)}^r = S_{(a,\rho)}^r=\tau^r$, as required.
\end{proof}

\bibliographystyle{abbrv}
\bibliography{bibliography}

\begin{thebibliography}{10}

\bibitem{Ackelsberg2024}
E.~Ackelsberg.
\newblock Khintchine‑type double recurrence in abelian groups.
\newblock {\em Ergodic Theory and Dynamical Systems}, 45(1):1--33, 2024.

\bibitem{ABB}
E.~Ackelsberg, V.~Bergelson, and A.~Best.
\newblock Multiple recurrence and large intersections for abelian group actions.
\newblock {\em Discrete Anal.}, pages Paper No. 18, 91, 2021.

\bibitem{ABS}
E.~Ackelsberg, V.~Bergelson, and O.~Shalom.
\newblock Khintchine-type recurrence for 3-point configurations.
\newblock {\em Forum Math. Sigma}, 10:Paper No. e107, 57, 2022.

\bibitem{Bergelsonidempotents}
V.~Bergelson.
\newblock Minimal idempotents and ergodic ramsey theory.
\newblock {\em Topics in dynamics and ergodic theory, London Math. Soc. Lecture Note Ser.}, 310:8--39, 2003.

\bibitem{bergelson2003minimal}
V.~Bergelson.
\newblock {Minimal idempotents and ergodic Ramsey theory}.
\newblock In {\em Topics in Dynamics and Ergodic Theory, London Math.~Soc.~Lecture Note Ser.}, volume 310. Cambridge Univ. Press, 2003.

\bibitem{leibman-hkz}
V.~Bergelson.
\newblock Host--{K}ra and {Z}iegler factors and convergence of multiple averages. with an appendices by {A}.~{L}eibman, and by {A}.~{Q}uas and {M}.~{W}ierdl.
\newblock In B.~Hasselblatt and A.~Katok, editors, {\em Handbook of Dynamical Systems}, volume~1B, pages 841--855. Elsevier, 2005.

\bibitem{bergelson2014rigidity}
V.~Bergelson, A.~del Junco, M.~Lema{\'n}czyk, and J.~Rosenblatt.
\newblock Rigidity and non-recurrence along sequences.
\newblock {\em Ergodic Theory Dynam. Systems}, 34(5):1464--1502, 2014.

\bibitem{BHK}
V.~Bergelson, B.~Host, and B.~Kra.
\newblock Multiple recurrence and nilsequences.
\newblock {\em Invent. Math.}, 160(2):261--303, 2005.
\newblock With an appendix by Imre Ruzsa.

\bibitem{bergelson2000aspects}
V.~Bergelson, B.~Host, R.~McCutcheon, and F.~Parreau.
\newblock aspects of uniformity in recurrence.
\newblock {\em Colloquium Mathematicae}, 84/85(2):549--576, 2000.

\bibitem{bergelson1996polynomial}
V.~Bergelson and A.~Leibman.
\newblock {Polynomial extensions of van der Waerden's and Szemer{\'e}di's theorems}.
\newblock {\em J.~Amer.~Math.~Soc.}, 9(3):725--753, 1996.

\bibitem{BergelsonIPSzemeredi}
V.~Bergelson and M.~Randall.
\newblock An ergodic {IP} polynomial {S}zemerédi theorem.
\newblock {\em Memoirs of the AMS}, 146, 2000.

\bibitem{Boshernitzan2005sequences}
M.~Boshernitzan, G.~Kolesnik, A.~Quas, and M.~Wierdl.
\newblock Ergodic averaging sequences.
\newblock {\em J. Anal. Math.}, 95:63--103, 2005.

\bibitem{cgss-2023}
P.~Candela, D.~Gonz\'alez-S\'anchez, and B.~Szegedy.
\newblock Lie-fibered nilspaces, double-coset nilspaces, and gowers norms.
\newblock {\em Preprint}, 2023.

\bibitem{candela-szegedy-inverse}
P.~Candela and B.~Szegedy.
\newblock Regularity and inverse theorems for uniformity norms on compact abelian groups and nilmanifolds.
\newblock {\em {J. f\"ur die Reine und Angew. Math.}}, 789:1--42, 2022.

\bibitem{candela-szegedy-memoir}
P.~Candela and B.~Szegedy.
\newblock Nilspace factors for general uniformity seminorms, cubic exchangeability and limits.
\newblock {\em Mem. Amer. Math. Soc.}, 287(1425):v+101, 2023.

\bibitem{MC3}
A.~Ferr\'{e}~Moragues.
\newblock Properties of multicorrelation sequences and large returns under some ergodicity assumptions.
\newblock {\em Discrete Contin. Dyn. Syst.}, 41(6):2809--2828, 2021.

\bibitem{Franpoly}
N.~Frantzikinakis.
\newblock Multiple ergodic averages for three polynomials and applications.
\newblock {\em Trans. Amer. Math. Soc.}, 360(10):5435--5475, 2008.

\bibitem{Fran2010}
N.~Frantzikinakis.
\newblock Multiple recurrence and convergence for {H}ardy sequences of polynomial growth.
\newblock {\em J. Anal. Math.}, 112:79--135, 2010.

\bibitem{FHK}
N.~Frantzikinakis, B.~Host, and B.~Kra.
\newblock The polynomial multidimensional {S}zemer\'{e}di theorem along shifted primes.
\newblock {\em Israel J. Math.}, 194(1):331--348, 2013.

\bibitem{FrantzikinakisKuca2025}
N.~Frantzikinakis and B.~Kuca.
\newblock Joint ergodicity for commuting transformations and applications to polynomial sequences.
\newblock {\em Inventiones Mathematicae}, 239(2):621--706, 2025.

\bibitem{furstenberg1977ergodic}
H.~Furstenberg.
\newblock {Ergodic behaviour of diagonal measures and a theorem of Szemer{\'e}di on arithmetic progressions}.
\newblock {\em J. Anal. Math.}, 31:204--256, 1977.

\bibitem{furstenberg2014recurrence}
H.~Furstenberg.
\newblock {\em Recurrence in Ergodic Theory and Combinatorial Number Theory}.
\newblock Princeton Legacy Library. Princeton University Press, 2014.

\bibitem{furstenberg1978ergodic}
H.~Furstenberg and Y.~Katznelson.
\newblock {An ergodic Szemer{\'e}di theorem for commuting transformations}.
\newblock {\em J. Anal. Math.}, 34:275--291, 1978.

\bibitem{FK85}
H.~Furstenberg and Y.~Katznelson.
\newblock An ergodic {S}zemer\'{e}di theorem for {IP}-systems and combinatorial theory.
\newblock {\em J. Analyse Math.}, 45:117--168, 1985.

\bibitem{furstenberg1991density}
H.~Furstenberg and Y.~Katznelson.
\newblock {A density version of the Hales-Jewett theorem}.
\newblock {\em J.~Anal.~Math.}, 57:64--119, 1991.

\bibitem{fw}
H.~Furstenberg and B.~Weiss.
\newblock A mean ergodic theorem for {$(1/N)\sum^N_{n=1}f(T^nx)g(T^{n^2}x)$}.
\newblock In {\em Convergence in ergodic theory and probability ({C}olumbus, {OH}, 1993)}, volume~5 of {\em Ohio State Univ. Math. Res. Inst. Publ.}, pages 193--227. de Gruyter, Berlin, 1996.

\bibitem{furstenberg1996mean}
H.~Furstenberg and B.~Weiss.
\newblock A mean ergodic theorem for {$(1/N)\sum^N_{n=1}f(T^nx)g(T^{n^2}x)$}.
\newblock In {\em Convergence in ergodic theory and probability ({C}olumbus, {OH}, 1993)}, volume~5 of {\em Ohio State Univ. Math. Res. Inst. Publ.}, pages 193--227. de Gruyter, Berlin, 1996.

\bibitem{gmv2}
Y.~Gutman, F.~Manners, and P.~P. Varj\'{u}.
\newblock The structure theory of nilspaces {II}: {R}epresentation as nilmanifolds.
\newblock {\em Trans. Amer. Math. Soc.}, 371(7):4951--4992, 2019.

\bibitem{gmv}
Y.~Gutman, F.~Manners, and P.~P. Varj\'{u}.
\newblock The structure theory of nilspaces {I}.
\newblock {\em J. Anal. Math.}, 140(1):299--369, 2020.

\bibitem{gmv3}
Y.~Gutman, F.~Manners, and P.~P. Varj\'{u}.
\newblock The structure theory of nilspaces {III}: {I}nverse limit representations and topological dynamics.
\newblock {\em Adv. Math.}, 365:107059, 53, 2020.

\bibitem{host2005nonconventional}
B.~Host and B.~Kra.
\newblock {Nonconventional ergodic averages and nilmanifolds}.
\newblock {\em Ann.~Math.}, 161(1):397--488, 2005.

\bibitem{hk-book}
B.~Host and B.~Kra.
\newblock {\em Nilpotent structures in ergodic theory}, volume 236 of {\em Mathematical Surveys and Monographs}.
\newblock American Mathematical Society, Providence, RI, 2018.

\bibitem{jst}
A.~Jamneshan, O.~Shalom, and T.~Tao.
\newblock The structure of arbitrary {C}onze-{L}esigne systems.
\newblock {\em Commun. Am. Math. Soc.}, 4:182--229, 2024.

\bibitem{jt21-1}
A.~Jamneshan and T.~Tao.
\newblock The inverse theorem for the {$U^3$} {G}owers uniformity norm on arbitrary finite abelian groups: {F}ourier-analytic and ergodic approaches.
\newblock {\em Discrete Anal.}, pages Paper No. 11, 48, 2023.

\bibitem{leibmanpolymap}
A.~Leibman.
\newblock Polynomial mappings of groups.
\newblock {\em Israel Journal of Mathematics}, 129:29--60, 2002.

\bibitem{lindenstraussdistal}
E.~Lindenstrauss.
\newblock Measurable distal and topological distal systems.
\newblock {\em Ergodic Theory and Dynamical Systems.}, 19(4):1063--1076, 1999.

\bibitem{shalom2}
O.~Shalom.
\newblock Multiple ergodic averages in abelian groups and {K}hintchine type recurrence.
\newblock {\em Trans. Amer. Math. Soc.}, 375(4):2729--2761, 2022.

\bibitem{shalom3}
O.~Shalom.
\newblock Host-{K}ra theory for {$\bigoplus_{p\in P} \Bbb F_p$}-systems and multiple recurrence.
\newblock {\em Ergodic Theory Dynam. Systems}, 43(1):299--360, 2023.

\bibitem{szemeredi1975sets}
E.~Szemer{\'e}di.
\newblock {On sets of integers containing no $k$ elements in arithmetic progression}.
\newblock {\em Acta.~Arith.}, 27:199--245, 1975.

\bibitem{walsh2012norm}
N.~M. Walsh.
\newblock {Norm convergence of nilpotent ergodic averages}.
\newblock {\em Ann. of Math.}, 175(3):1667--1688, 2012.

\bibitem{wooley}
T.~D. Wooley and T.~D. Ziegler.
\newblock Multiple recurrence and convergence along the primes.
\newblock {\em Amer. J. Math.}, 134(6):1705--1732, 2012.

\bibitem{zieglerformula}
T.~Ziegler.
\newblock A non-conventional ergodic theorem for a nilsystem.
\newblock {\em Ergodic Theory and Dynamical Systems}, 25 no. 4:1357--1370., 2005.

\bibitem{ziegler2007universal}
T.~Ziegler.
\newblock {Universal characteristic factors and Furstenberg averages}.
\newblock {\em J.~Amer.~Math.~Soc.}, 20:53--97, 2007.

\end{thebibliography}

\end{document}